\documentclass[oneside,notitlepage,12pt]{article}

\usepackage{amssymb}
\usepackage[leqno]{amsmath}
\usepackage{amsfonts}
\usepackage{amsopn}
\usepackage{amstext}
\usepackage{amsthm}
\usepackage{enumitem}

\usepackage{tikz}
\usetikzlibrary{arrows,topaths}
\usetikzlibrary{shapes}
\usetikzlibrary{cd}
\usetikzlibrary{decorations.pathmorphing}
\newcommand{\nx}{\pgfmatrixnextcell}

\usepackage[colorlinks, linkcolor=blue, backref]{hyperref}

\usepackage{calrsfs}

\usepackage{clock} 

\providecommand{\cal}{\mathcal}
\renewcommand{\Bbb}{\mathbb}

\newenvironment{pf}{\begin{proof}}{\end{proof}}

\newcommand{\Ef}{{\cal{F}}}

\newcommand{\Nat}{{\Bbb{N}}}

\newcommand{\sig}{\sigma}

\renewcommand{\phi}{\varphi}
\renewcommand{\rho}{\varrho}

\newcommand{\rest}{\restriction}

\newcommand{\ntr}{{n\in\omega}}

\newcommand{\loe}{\leqslant}
\newcommand{\goe}{\geqslant}

\newcommand{\subs}{\subseteq}
\newcommand{\sups}{\supseteq}
\newcommand{\nnempty}{\ne\emptyset}

\newcommand{\ovr}{\overline}
\renewcommand{\iff}{\Longleftrightarrow}

\newcommand{\cl}{\operatorname{cl}}

\newcommand{\diam}{\operatorname{diam}}

\newcommand{\id}[1]{{\operatorname{i\!d}_{#1}}} 

\newcommand{\concat}{{}^\smallfrown}

\newcommand{\setof}[2]{\{#1\colon #2\}}

\newcommand{\sett}[2]{\{#1\}_{#2}}
\newcommand{\sn}[1]{\{#1\}} 
\newcommand{\dn}[2]{\{#1,#2\}} 
\newcommand{\pair}[2]{(#1, #2)} 

\newcommand{\map}[3]{#1\colon #2 \to #3} 
\newcommand{\img}[2]{#1[#2]} 

\newcommand{\fin}[1]{[#1]^{<\omega}}

\newcommand{\fra}{Fra\"iss\'e}

\providecommand{\nat}{\omega}

\newcommand{\ciag}[1]{{\sett{{#1}_n}{\ntr}}}

\newcommand{\iso}{\approx}

\newcommand{\ciagi}[1]{\sig{#1}}

\newcommand{\cmp}{\circ} 

\newcommand{\separator}{\begin{center}***\end{center}}

\newcommand{\proto}[1]{{\mathbb S_\kappa}}

\newtheorem{tw}{Theorem}[section]
\newtheorem{wn}[tw]{Corollary}
\newtheorem{lm}[tw]{Lemma}
\newtheorem{prop}[tw]{Proposition}

\theoremstyle{definition}
\newtheorem{df}[tw]{Definition}
\newtheorem{example}[tw]{Example}

\newtheorem{question}[tw]{Question}
\newtheorem{uwgi}[tw]{Remark}

\newcommand{\age}{\operatorname{Age}}

\newcommand{\fG}{\mathfrak G} 

\newcommand{\catligraphs
}{\fG_{\rightarrowtail}} 
\newcommand{\catrigraphs
}{\fG^{\twoheadleftarrow}} 
\newcommand{\catepgraphs
}{\fG^{\twoheadleftarrow}_{\rightarrowtail}} 

\newcommand{\rgraph}{{\mathbf R}} 
\newcommand{\PG}{\mathbb P} 
\newcommand{\Hen}[1]{\mathbf H_{#1}}

\newcommand{\ufr}{{\mathbf U_{\operatorname{FR}}}}
\newcommand{\ufrctd}{{\mathbf U_{\operatorname{FR}}^{\circ}}}

\newcommand{\uppr}{{\mathbf U_{\operatorname{PPR}}}} 
\newcommand{\upprctd}{{\mathbf U_{\operatorname{PPR}}^{\circ}}} 

\newcommand{\pufr}{{\ovr{\mathbf U}_{\operatorname{FR}}}}
\newcommand{\pufrctd}{{\ovr{\mathbf U}_{\operatorname{FR}}^{\circ}}}

\newcommand{\puppr}{{\ovr{\mathbf U}_{\operatorname{PPR}}}} 

\newcommand{\pupprctd}{{\ovr{\mathbf U}_{\operatorname{PPR}}^{\circ}}} 

\newcommand{\soc}{{\mathbf S}} 
\newcommand{\psoc}{\ovr{\mathbf S}} 

\newcommand{\origin}{\Theta} 

\newcommand{\ewa}[3]{\begin{tikzcd}{#2} \arrow[r, rightsquigarrow, "{#1}"] \nx {#3} \end{tikzcd}}

\title{Evolutions of finite graphs}
\author{
{\sc Stefan Geschke}\\
{\small Department of Mathematics}\\
{\small University of Hamburg}\\
{\small Bundesstr. 55 (Geomatikum)}\\
{\small 20146 Hamburg, Germany}
\and
{\sc Szymon G{\l}\c{a}b}\\
{\small Lodz University of Technology}\\ 
{\small W\'{o}lcza\'nska 215, 93-005 \L\'od\'z, Poland}
\and
{\sc Wies{\l}aw Kubi\'s}\footnote{Research supported by project EXPRO 20-31529X (Czech Science Foundation).}\\
{\small Institute of Mathematics, Czech Academy of Sciences}\\
{\small Žitná 25, 115 67 Praha 1, Czech Republic}
}
\date{\today\ \clocktime}

\begin{document}

\maketitle

\begin{abstract}
	Every countable graph can be built from finite graphs by a suitable infinite process, either adding new vertices randomly or imposing some rules on the new edges. On the other hand, a profinite topological graph is built as the inverse limit of finite graphs with graph epimorphisms.
	We propose to look at both constructions simultaneously. We consider countable graphs that can be built from finite ones by using both embeddings and projections, possibly adding a single vertex at each step. We show that the Rado graph can be built this way, while Henson's universal triangle-free graph cannot.
	We also study the corresponding profinite graphs. Finally, we present a concrete model of the projectively universal profinite graph, the projective \fra\ limit of finite graphs, showing in particular that it has a dense subset of isolated vertices.
	
	\ 
	
	
\end{abstract}

\tableofcontents

\section{Introduction}

An infinite graph, or any other first order mathematical structure, is often presented in the form $G = \bigcup_{\ntr}G_n$, where $\ciag{G}$ is a chain of finite substructures, namely, $G_n \subs G_{n+1}$ and all the relations in $G_n$ are induced from $G_{n+1}$. Perhaps the most precise way of representing $G$ is requiring that $G_{n+1} \setminus G_n$ contains a single vertex for each $\ntr$. Obviously, every countable graph can be built this way and there is nothing more to say about it. On the other hand, one may impose some conditions on the inclusions $G_n \subs G_{n+1}$. It turns out that requiring $G_n$ to be a retract of $G_{n+1}$ for each $n$, restricts the class of graphs significantly.

As one of the motivations, consider a finite graph $G$ representing the acquaintance of members of a given group of people (e.g. a social network, a scientific society). Such a graph evolves in time, at a single step adding a new vertex representing a new member $w$ that was accepted by a fixed member $v$ of $G$. By this way, mapping $w$ to $v$ and keeping the identity on $G$ is a graph homomorphism, assuming all the members of $G$ that are ``friends'' with $w$ have already been known by $v$ (who in fact can have more ``friends'').
An extreme case is when $G$ has a member $v$ connected to all other members (we will later call it a \emph{sentinel}). In that case, $v$ can ``invite'' anybody, no matter what the adjacency relation to a new member is. In other words, for every graph $G' \sups G$, the mapping $\map r{G'}G$ defined by $r(x) = v$ for $x \in G' \setminus G$ and $r(x) = x$ for $x \in G$, is a retraction. We shall consider graphs with loops, so that edges can be collapsed by homomorphisms. Without this assumption, a complete graph would admit no retractions except the identity, while we would like to see every of its induced subgraphs as a retract.

We are going to look at graphs that are inverse limits of nonempty finite graphs, typically called \emph{profinite}.
In fact, profinite graphs have been studied in various contexts by several authors, see e.g. \cite{GilRib}, \cite{AuiSte}.
We restrict attention to undirected graphs with unique edges, although some ideas can be easily adapted to the more general case of directed graphs, possibly with multiple arrows.

It is natural to ask for a projective counterpart of the random graph, that is, a graph represented as the inverse limit of finite graphs which would be ``projectively homogeneous'' with respect to finite graphs
(of course, in this setting we have to exclude the empty graph).
It turns out that such a graph indeed exists, which actually follows from the general theory of \fra\ limits (see \cite{DrGoe92}, \cite{IrwSol}, \cite{K_flims}).
We discuss it in Section~\ref{SecPGebrgihbe}.
From the graph-theoretic point of view, the projective random graph is not extremely interesting, as it does not contain paths of length $>1$.
In particular, adjacency is a closed equivalence relation and both the graph and its quotient with respect to this relation are homeomorphic to the Cantor set (Proposition~\ref{TheoremCamerlo} below, already proved by Camerlo~\cite{Camerlo}).

That is why we propose a new natural way of looking at compact graphs.
Namely, we consider both embeddings and projections at the same time.
To be more precise, we consider graphs that can be obtained as inverse limits of finite graphs, where all bonding maps are right-invertible in the category of graphs (we assume that each vertex is adjacent to itself).
This leads to several interesting examples, which are described in Section~\ref{SectionPrelim} below.

The theory of \fra\ sequences and limits was primarily founded in model theory, and then generalized to the framework of category theory. We want our paper to be easily accessible to readers interested in graph theory, and therefore we decided to avoid category-theoretic and model-theoretic notation, whenever possible. We refer to \cite{K_flims} for a precise categorical approach to \fra\ theory.

The paper is organized as follows. In the preliminary section we start from a brief description of the random graph construction. The random graph can be viewed as a \fra\ limit in the class (or category) of finite graphs with embeddings. Then we consider the class of all finite graphs with quotient mappings. Limits of inverse sequences turn out to be compact topological graphs on zero-dimensional metrizable spaces, or closed graphs on closed subsets of the Cantor space $2^\omega$; they are typically called \emph{profinite graphs} (we exclude the empty graph here). Next, we pass to retractions of graphs. A sequence of finite graphs with both embeddings and retractions has the limit consisting of two graphs: a profinite graph and its countable dense induced subgraph. We call the latter one \emph{finitely retractable}. Finally, we introduce \emph{graph evolutions}, namely, processes of adding a new vertex together with a retraction; graphs arising from such evolutions will be called \emph{point-by-point retractable} and we exhibit a special subclass of \emph{sociable graphs}. All these classes of mappings lead to different universal graphs. Their properties, which follow from general categorical \fra\ theory, are presented. In the preliminary section we also show topological properties of the set of all isolated vertices in profinite graphs and we show that compact graphs cannot be $\omega$-saturated. In Section \ref{SecPGebrgihbe} we show the properties of projectively universal profinite graphs which can be deduced from \fra\ theory. In Section \ref{SectionModelQuotient} we construct the model of such a graph.
In Section \ref{SectionRandomGraph} we show how the random graph is related to the notions defined in the previous sections. More precisely, we show that the random graph is sociable, and it can be embedded into a closed graph on the Cantor space as a dense subspace. It does not immediately follow from the definition whether there are countable graphs which are not finitely retractable. In Section \ref{SectionHensonGraphs} we show that Henson's universal $K_n$-free graphs are such examples. It can be easily shown that the cycle $C_5$ is not point-by-point retractable and $C_4$ is not sociable. These simple examples show that the class of sociable graphs is a proper subclass of point-by-point retractable graphs, which in turn is a proper subclass of finitely retractable graphs. We show, in particular, that finite retracts of sociable graphs are sociable. As a corollary, we obtain that the universal finitely retractable graph $\ufr$ is not sociable, as every finite graph is its retract.

\section{Preliminaries}\label{SectionPrelim}

Throughout this note we assume that the edge relation of a graph is reflexive, in other words, every vertex has a loop. Our assumption is necessary in the study of graph \emph{homomorphisms}, namely, edge-preserving mappings. For example, we would like the one-element graph to be a homomorphic image of any graph, while not allowing loops a graph homomorphism would be one-to-one on every complete subgraph. Below we formally recall the necessary definitions.

By a \emph{graph} we mean a structure of the form $G = \pair V E$, where $E$ is a symmetric and reflexive binary relation on $V$.
As usual, the elements of $E$ are called \emph{edges} and the elements of $V$ are called \emph{vertices}.
We say $x,y$ are \emph{adjacent} if $\pair xy \in E$.
We shall often identify $G$ and $V$.
A \emph{homomorphism of graphs} is a map $\map f G H$ that preserves edges. Note that graph homomorphisms can map non-adjacent pairs to adjacent pairs. Furthermore, reflexivity allows for collapsing any subset of a graph to a single vertex.

A quite common definition of a (simple) graph is a structure of the form $\pair{V}{E}$, where $E$ is a collection of precisely 2-element subsets of $V$, which is equivalent to an irreflexive symmetric binary relation. In that case, homomorphisms are not allowed to collapse edges. On the other hand, from the model-theoretic point of view there is no difference, as both structures (the reflexive and irreflexive one) are inter-definable, exactly in the same way as strict partial orderings are inter-definable with reflexive partial orderings. The only practical difference is the notion of a homomorphism.

Given a graph $G$, an \emph{induced subgraph} is a subset $H \subs G$ with the induced graph relation, namely, $x E y$ in $H$ if an only if $x E y$ in G. A \emph{subgraph} often refers to a subset $A$ of $G$ such that the inclusion $A$ is a graph homomorphism.

\subsection{Embeddings, quotients, retractions}

An \textit{embedding} $\map eGH$ is a one-to-one graph homomorphism such that whenever $\pair {e(x)}{e(y)}$ is adjacent in $H$, $\pair xy$ is adjacent in $G$ as well. That is, an embedding $\map eGH$ is an isomorphism between $G$ and the induced subgraph $\img eG$ of $H$.

A graph homomorphism $\map fXY$ is \textit{strict} if for every adjacent $p,q$ in the range of $f$ there exist adjacent $a,b\in X$ such that $f(a)=p$ and $f(b)=q$. A strict homomorphism which is additionally surjective will be called a \textit{quotient map}.
 
Let us present the following instructive example of a quotient map. Let $\map pHG$ be a mapping defined as in the picture---the nodes in $H$ are mapped onto the nodes in $G$ with the same symbols: 
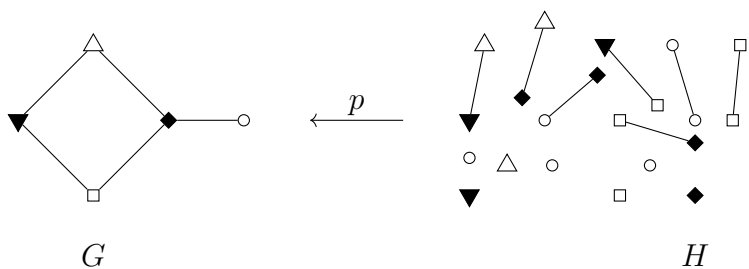
\[
\begin{tikzpicture}[every node/.style={circle,inner sep=1.5pt,fill=white}]
  \node[regular polygon,regular polygon sides=3, draw, shape border rotate=180,fill=black] (A) at (0,1) {};
  \node[regular polygon,regular polygon sides=3, draw, shape border rotate=0](B) at (1,2) {};
  \node [regular polygon,regular polygon sides=4, draw, shape border rotate=0](C) at (1,0) {};
  \node [regular polygon,regular polygon sides=4, draw, shape border rotate=45,fill=black] (D) at (2,1) {};
  \node [circle, draw, shape border rotate=0] (E) at (3,1) {};
\node at (1,-0.8) {$G$};  

  \draw (A) -- (B)
        (B) -- (D)
        (C) -- (D)
        (D) -- (E)
        (A) -- (C);
        
 \node[regular polygon,regular polygon sides=3, draw, shape border rotate=180,fill=black] (1) at (6,1) {};       
   \node[regular polygon,regular polygon sides=3, draw, shape border rotate=0](2) at (6.2,2) {};
   \node[regular polygon,regular polygon sides=3, draw, shape border rotate=0](3) at (7,2.3) {};
\node [regular polygon,regular polygon sides=4, draw, shape border rotate=45,fill=black] (4) at (6.7,1.3) {};
\node [regular polygon,regular polygon sides=4, draw, shape border rotate=45,fill=black] (5) at (7.7,1.6) {};
  \node [circle, draw, shape border rotate=0] (6) 
at (7,1) {};
\node[regular polygon,regular polygon sides=3, draw, shape border rotate=180,fill=black] (7) at (7.8,2) {};     
  \node [regular polygon,regular polygon sides=4, draw, shape border rotate=0](8) at (8.5,1.2) {};
\node [regular polygon,regular polygon sides=4, draw, shape border rotate=0](9) at (8,1) {};
\node [regular polygon,regular polygon sides=4, draw, shape border rotate=45,fill=black](10) at (9,0.7) {};

\node[regular polygon,regular polygon sides=3, draw, shape border rotate=180,fill=black] (11) at (6,0) {};
  \node[regular polygon,regular polygon sides=3, draw, shape border rotate=0](12) at (6.5,0.4) {};
  \node [regular polygon,regular polygon sides=4, draw, shape border rotate=0](13) at (8,0) {};
  \node [regular polygon,regular polygon sides=4, draw, shape border rotate=45,fill=black] (14) at (9,0) {};
  \node [circle, draw, shape border rotate=0] (15) at (7.1,0.4) {};
  \node [circle, draw, shape border rotate=0] (20) at (6,0.5) {};
  \node [circle, draw, shape border rotate=0] (21) at (8.4,0.4) {};
    \node [circle, draw, shape border rotate=0] (16) at (9,1) {};
      \node [circle, draw, shape border rotate=0] (17) at (8.7,2) {};
  \node [regular polygon,regular polygon sides=4, draw, shape border rotate=0](18) at (9.5,1) {};
    \node [regular polygon,regular polygon sides=4, draw, shape border rotate=0](19) at (9.6,2) {};
  \node (p1) at (5.2,1) {};
  \node (p2) at (3.8,1) {};
  \node at (4.5,1.2) {$p$};

  \node at (9,-0.8) {$H$};  
\draw (1) -- (2)
	  (3) -- (4)
	  (5) -- (6)
	  (7) -- (8)
	  (9) -- (10)
	  (16) -- (17)
	  (18) -- (19);
\draw[->] (p1) -- (p2);
\end{tikzpicture}
\]
The graph $H$ consists only of isolated nodes, drawn at the bottom, and paths of length $1$. Note that $p$ is a quotient map. This example shows that any graph is a quotient of some fragmented graph---that is, a graph every component of which has at most 2 nodes. Note also that if we remove from $H$ the edge\; 
\begin{tikzpicture}[every node/.style={circle,inner sep=1.5pt,fill=white}]
  \node[regular polygon,regular polygon sides=3, draw, shape border rotate=180,fill=black] (A) at (0,0) {};
  \node[regular polygon,regular polygon sides=3, draw, shape border rotate=0](B) at (1,0) {};
  \draw (A) -- (B);
\end{tikzpicture}
, then $p$ remains a homomorphism that is not strict.

\subsection{Inverse limits and profinite graphs}

Consider the following sequence of finite graphs
\[
\begin{tikzcd}
G_0&\arrow[l, "q_{0}"'] G_1&\arrow[l, "q_{1}"']G_2&\arrow[l,"q_{2}"']\cdots
\end{tikzcd}
\]  
where $\map{q_{i}}{G_{i+1}}{G_i}$ is a quotient mapping. Such a sequence has the limit called the \emph{inverse limit} $G=\varprojlim\langle (G_n)_{n\in\omega},(q_{k}^n)_{k\leq n}\rangle$, where $q^n_k=q_k\circ q_{k+1}\circ\cdots\circ q_{n-1}$ maps $G_n$ onto $G_k$ ($q^n_k$ is a quotient as a composition of quotients). More precisely $G$ is a graph on the set $V:=\{(x_i)\in\prod_{i\in\omega}V_i:q_{i}(x_{i+1})=x_i\}$ and its graph relation is given by
\[
(x_n)E(y_n)\iff \forall n\in\omega(x_n E_n y_n).
\] 
If we equip each $V_i$ with the discrete topology, then the product $\prod V_i$ admits the product topology, which is given by basic open sets of the form
\[
U_{(t_0,t_1,\dots,t_{k-1})}:=\{t_0\}\times\{t_1\}\times\dots\times\{t_{k-1}\}\times V_{k}\times V_{k+1}\times\dots
\]
for any $k\in\omega$ and any $t_i\in V_i$. With this product topology, $\prod V_i$ becomes a zero-dimensional compact metrizable space. The space $V$ as a closed subset of $\prod V_i$ is a zero-dimensional compact metrizable space as well, and $E$ as a subset of $V\times V$ is closed, in other words $E$ is a closed relation. If the topological space $V$ has no isolated points, then $V$ is homeomorphic to the Cantor space $2^\omega$. This happens if and only if for any $k\in\omega$ and $x\in V_k$ there is $n>k$ and two distinct $y,z\in V_n$ with $q^n_k(y)=q^n_k(z)=x$.

A topological graph $H= \pair {V_H}{E_H}$ is called a \emph{profinite graph} if $V_H \nnempty$ and there is $\map fGH$  which is a graph isomorphism and at the same time a homeomorphism of $V$ and $V_H$ for some inverse limit $G=\pair VE$ of finite graphs. We already know that profinite graphs are (nonempty) compact graphs on zero-dimensional metrizable compact spaces. It turns out that this actually characterizes profinite graphs.  
 
\begin{prop}
Let $E_H$ be a symmetric and reflexive closed relation on a nonempty zero-dim\-en\-sional metrizable compact space $V_H$. Then $H=\pair {V_H}{E_H}$ is a profinite graph. 
\end{prop} 

\begin{proof}
Each zero-dimensional metrizable compact space is homeomorphic to a closed subset of the Cantor space $\{0,1\}^\omega$. Therefore we may assume that $V_H$ is a closed subset of $\{0,1\}^\omega$. For $n\geq 1$ let $\map {p_n}{V_H}{\{0,1\}^n}$ be given by $$p_n(x)=\big(x(0),\dots,x(n-1)\big).$$
Let $V_n:=p_n(V_H)$ and 
\[
\big(a(0),\dots,a(n-1)\big)E_n\big(b(0),\dots,b(n-1)\big) \iff 
\]
\[
\exists x,y\in V_H
 \Big(xE_Hy \wedge\forall i<n \big(x(i)=a(i)\wedge y(i)=b(i)\big)\Big). 
\]
Let $G_n:=(V_n,E_n)$. By the definition of $E_n$, $p_n$ is a quotient map. For $k\leq n$ let $p_k^n$ denote the projection from $\{0,1\}^{n}$ to $\{0,1\}^k$ given by $$p_{k}^n\big(x(0),\dots,x(n-1)\big)=\big(x(0),\dots,x(k-1)\big).$$
Let $V:=\{(t_n)\in\prod_{n\geq 1} V_n:(\forall\; n)\;\; p^{n+1}_{n}(t_{n+1})=t_n\}$ and define a graph relation $E$ by
\[
(t_n)E(s_n)\iff\forall k(t_k E_k s_k).
\] 
Then $G:=(V,E)$ is an inverse limit of $\pair {G_n}{p_{k}^n}$. Let $\map f {V_H}{\prod_{n\geq 1} V_n}$ be given by $f(x)=(x_1,x_2,x_3,\dots)$ where $x_n=\big(x(0),\dots,x(n-1)\big)$. Note that $p^{n+1}_{n}(x_{n+1})=x_n$ which means that $f$ maps $V_H$ to $V$. Note that $\map gV{V_H}$ given by $g(t_1,t_2,t_3,\dots)=\big(t_1(0),t_2(1),t_3(2),\dots\big)$ is the inverse of $f$, and for any $x,y\in V_H$,
\[
xE_Hy\iff\forall n(x_n E_n y_n)\iff f(x)E f(y).
\]
Thus $f$ is a graph-isomorphism. Moreover
\begin{align*}
	f\Big(V_H\cap\big(\{x(0)\}\times & \dots\times \{x(k-1)\}\times\{0,1\}\times\{0,1\}\times\dots\big)\Big) \\ &=
	\{x_1\}\times\dots\times\{x_k\}\times V_{k+1}\times V_{k+2}\times\dots
\end{align*}
which shows that $f$ is a homeomorphism of $V_H$ and $V$. Therefore $H$ is a profinite graph.
\end{proof}


\subsubsection{Isolated vertices}

A vertex $v$ in a graph $G$ is \emph{isolated} if it is adjacent to itself only. This should not be confused with a topologically isolated point.

\begin{lm}\label{Gdelta}
	Let $G$ be a graph on a compact metric space $(X,d)$ with a closed edge relation $E$.
	Then the set of isolated vertices of $G$ is a $G_\delta$ subset of $X$.  
\end{lm}

\begin{proof}
	Let $I$ denote the set of isolated vertices of the graph $G$.
	For each vertex $x\in I$ and $\varepsilon>0$ we define a real number $\delta(x,\varepsilon)$ as follows.
	
	The set $Y=\{y\in X:d(x,y)\geq\varepsilon\}$ is closed.
	Since $x$ is isolated and the edge relation $E$ is closed, for each $y\in Y$ there are an open neighborhood of $U_y$ of $y$ and $\delta_y>0$ such that
	$$\forall x',y'\in X:(d(x,x')<\delta_y\wedge y'\in U_y)\Rightarrow (x',y')\not\in E.$$
	Since $Y$ is compact, there are $n$ and $y_1,\dots,y_n\in Y$ such that $$Y\subseteq U_{y_1}\cup\dots\cup U_{y_n}.$$
	Let $\delta(x,\varepsilon)=\min(\varepsilon,\delta_{y_1},\dots,\delta_{y_n}).$
	Now, if $x'\in X$ has distance less than $\delta(x,\varepsilon)$ from $x$ and $y\in Y$, then $(x',y)\not\in E$.
	
	For each $\varepsilon>0$ let 
	$$V_\varepsilon=\{x'\in X:\exists x\in I \; (d(x,x')<\delta(x,\varepsilon))\}.$$
	Clearly, $V_\varepsilon$ is open.
	We claim that $$I=\bigcap_{n=1}^\infty V_{1/n}.$$
	
	It is clear that $I\subseteq\bigcap_{n=1}^\infty V_{1/n}.$
	On the other hand, if $y\in X\setminus I$, then there is $z\in X$ such that $(y,z)\in E$.  
	Let $n\in\mathbb N$ be such that $2/n<d(x,y)$.
	We show that $y\not\in V_{1/n}$.
	
	Suppose for a contradiction that $y\in V_{1/n}$.
	Let $x\in I$ be such that $d(x,y)<\delta\left(x,1/n\right)\leq 1/n$.
	Since $d(y,z)>2/n$, by the triangle inequality $d(x,z)>1/n$.
	By the choice of $\delta\left(x,1/n\right)$, we obtain that $(y,z)\not\in E$, which yields a contradiction.
	
	This shows that $I=\bigcap_{n=1}^\infty V_{1/n}$, proving that $I$ is a $G_\delta$ subset of $X$.
\end{proof}

Let us remark that Lemma \ref{Gdelta} can be formulated in a more general context. Having a relation $E$ on $X$ let us say that $x\in X$ is \emph{$E$-isolated} if $(x,y)\notin E$ for any $y\neq x$. If $X$ is compact and $E$ is closed, then repeating the above proof we obtain that the set of $E$-isolated points is a $G_\delta$ subset of $X$. 

\begin{lm}\label{LMergre}
	Let $G=(V,E)$ be a profinite graph and let $x\in V$. Then the following two conditions are equivalent:
	\begin{itemize}[itemsep=0pt]
		\item[{\rm(1)}] $x$ is an isolated vertex;
		\item[{\rm(2)}] there is a decreasing sequence $\ciag U$ of clopen neighborhoods of $x$ such that $\diam(U_n)\to 0$ and
		\[
		(\forall n\in\omega)(\forall y\in V\setminus U_n)(\forall z\in U_{n+1})\;\;\neg(z E y).
		\]
	\end{itemize}
\end{lm} 

\begin{proof}
	If $x$ is not isolated, then there is $y\neq x$ connected with $x$. Take any decreasing sequence $\ciag U$ of clopen neighborhoods of $x$ with $\diam(U_n)\to 0$. Find $n$ such that $\diam(U_n)$ is strictly less than the distance from $x$ to $y$. Then $y\in V\setminus U_n$, $x\in U_{n+1}$ and $x E y$. This shows that there is no sequence $\ciag U$ fulfilling (2).
	
	Assume that $x$ is an isolated vertex. Let $U_0$ be any clopen neighborhood of $x$ with $\diam(U_0)<1$. Since $E$ is a closed subset of $V\times V$, for any $y\in V\setminus U_0$ there are $W^1_y$ and $W^0_y$ such that 
	\[
	(W_y^1\times W_y^0)\cap E=\emptyset,\;\; y\in W^1_y\text{ and }x\in W^0_y.
	\]
	Since $V\setminus U_0$ is compact, there are $y_1,\dots,y_k\in V\setminus U_0$ such that $\bigcup_{i=1}^k W_{y_i}^1=V\setminus U_0$. Put $U_1:=\bigcap_{i=1}^k W^0_{y_i}$. Then
	\[
	(\forall y\in V\setminus U_0)(\forall z\in U_1)\;\;\neg(zEy). 
	\]
	By shrinking $U_1$ if necessary, we may assume that $\diam (U_1)<1/2$. 
	
	Proceeding inductively, we can construct the required sequence $\ciag U$. 
\end{proof}

In the next results we abuse the notation by identifying a graph $G$ with its set of vertices. The edge (adjacency) relation is still denoted by $E$.

\begin{lm}\label{THMKriterionIsoVerticess}
	Let $G$ be a profinite graph with the following property.
	\begin{enumerate}
		\item[{\rm(IV)}]  For every nonempty open set $U \subs G$ there is a nonempty open subset $W \subs U$ such that there are no edges between $W$ and $G \setminus U$.		
	\end{enumerate}
	Then $G$ has a dense $G_\delta$ subset of isolated vertices.
\end{lm}

\begin{proof}
	Thanks to Lemma~\ref{Gdelta}, we only need to show density.
	Fix a nonempty open set $U \subs G$. By induction, we can construct a decreasing sequence of nonempty clopen sets $\ciag U$ such that $U_0 \subs U$, $\diam(U_n) \to 0$ with respect to a fixed metric, and there are no adjacencies between $U_{n+1}$ and $G \setminus U_n$ for every $\ntr$. By compactness, there is a (unique) point $v \in \bigcap_{\ntr}U_n$.
	Lemma~\ref{LMergre} says that $v$ is an isolated vertex.	
\end{proof}

\begin{lm}\label{Weakly3SaturatedProfiniteHasNoIsolatedPoints}
	Let $G$ be a profinite graph. Assume that $H$ is a dense induced subgraph of $G$ which is positively $3$-saturated, i.e. for any $X\subset H$, $\vert X\vert< 3$, there is $x\in H\setminus X$ connected to every vertex in $X$. Then $G$ contains no isolated vertices.  
\end{lm} 


\begin{proof}
	Suppose to the contrary that $x\in G$ is an isolated vertex. By Lemma~\ref{LMergre}, there is a sequence $\ciag U$ fulfilling (2). Passing to subsequence we may assume that $V\setminus U_0\neq \emptyset$, $U_{k}\setminus U_{k+1}\neq\emptyset$. Take any $a\in (V\setminus U_0)\cap H$ and $b\in U_2\cap H$. By the positive $3$-saturation there is $y\in H$ connected to $a$ and $b$. Now, if $y\in U_1$, then $a\in V\setminus U_0$ and $y\in U_1$ are connected, which is a contradiction. If $y\in V\setminus U_1$, then $y\in V\setminus U_1$ and $b\in U_2$ are connected which is again a contradiction.
\end{proof}

\subsection{\fra\ limits}

Fix a class $\Ef$ consisting of finite structures of a fixed first-order language. We agree that $\Ef$ is defined up to isomorphism, that is, if $X \in \Ef$ then every structure isomorphic to $X$ is in $\Ef$.
Assume that there are at most countably many isomorphic types in $\Ef$ and every two structures from $\Ef$ embed into a common one that is also in $\Ef$ (this is called the \emph{joint embedding property}). Furthermore, let us assume that $\Ef$ has the \emph{amalgamation property}, namely, for every two embeddings $\map fZX$, $\map gZY$ with $Z,X,Y \in \Ef$ there exist $W \in \Ef$ and embeddings $\map {f'}XW$, $\map {g'}YW$ such that $f' \cmp f = g' \cmp g$. Finally, let us assume that $\Ef$ is \emph{hereditary}, namely, closed under taking substructures. 

A theorem of \fra~\cite{Fraisse} says that there is a unique, up to isomorphism, countable structure $U$ (called the \emph{\fra\ limit of $\Ef$}) such that $\Ef$ is the class of all (isomorphic copies of) finite substructure of $U$ and moreover $U$ is \emph{homogeneous}\footnote{Some authors use the term \emph{ultra-homogeneity} instead of \emph{homogeneity}, in order to avoid confusion with transitivity.} in the sense that every isomorphism between finite substructures of $U$ extends to an automorphism of $U$.
Actually, $U$ is characterized by \emph{$\Ef$-injectivity}, that is:
\begin{enumerate}
	\item[(I)] For every $A \subs B$ with $A, B \in \Ef$, every embedding of $A$ into $U$ extends to an embedding of $B$.
\end{enumerate}
More precisely, a countable structure is the \fra\ limit of $\Ef$ if and only if it is $\Ef$-injective and $\Ef$ is its \emph{age}, namely, $\Ef$ coincides with the class $\age(U)$ of all structures isomorphic to finite substructures of $U$. An important property of the \fra\ limit $U$ is \emph{universality}, namely, every countable structure whose age is contained in $\Ef$ embeds into $U$.

Perhaps the most classical examples of \fra\ limits are the linearly ordered set of all rational numbers (where $\Ef$ is the class of all finite linear orderings) and the \emph{random graph}, where $\Ef$ is the class of all finite graphs. There are of course many other examples, see e.g. the survey~\cite{Macpherson}.

As it happens, the concept of a \fra\ limit is purely category-theoretic, namely, instead of a class of structures with embeddings one can take an arbitrary category whose arrows play the role of embeddings. This was formalized by Droste and G\"obel~\cite{DrGoe92} and developed later by the third author~\cite{K_flims}.
The joint embedding property is sometimes called \emph{directedness}. The amalgamation property is defined in the same way as above, replacing embeddings by arrows of the category. The requirement that there are countably many isomorphic types needs to be strengthened, either assuming that the category is locally countable (every hom-set is countable) or assuming that the category is dominated by a countable subcategory, see~\cite{K_flims} for details.

A particular case is \emph{projective \fra\ theory} developed by Irwin and Solecki~\cite{IrwSol}, where one deals with a class of nonempty finite structures together with quotient homomorphisms. In that case, homogeneity means lifting isomorphisms between finite quotients to automorphisms, while universality means that every suitable inverse limit structure is a quotient.
We shall discuss the projective \fra\ limit of finite graphs in Section~\ref{SecPGebrgihbe}.

\subsubsection{The random graph}

A natural object in graph theory is the \emph{random graph}, that is, the unique countable graph $\rgraph$ with the following property:
\begin{enumerate}
	\item[(R)] Given disjoint finite sets $A, B \subs \rgraph$, there exists a vertex in $\rgraph \setminus (A \cup B)$ that is connected to all the elements of $A$ and to none of the elements of $B$.
\end{enumerate}
The graph $\rgraph$ contains isomorphic copies of all countable graphs and is homogeneous in the sense that every isomorphism between finite subgraphs of $\rgraph$ extends to an automorphism of $\rgraph$.
Actually, the existence, universality and uniqueness of the random graph follows from the classical (model-theoretic) \fra\ theory~\cite{Fraisse}.
A natural question is whether there exist other countable homogeneous graphs, besides $\rgraph$, the discrete one and the complete one.
Using \fra\ theory, it is relatively easy to see that for each $n>2$ there exists a homogeneous $K_n$-free graph $\Hen n$ (probably first found explicitly and studied by Henson~\cite{Henson}). 
Finally, Lachlan and Woodrow~\cite{LachlanWoodrow} showed that, modulo taking the opposite graphs and disjoint unions of complete graphs, there are no other countable homogeneous graphs.

\subsubsection{No compact graph can be $\omega$-saturated}

A graph $G$ on a topological space $X$ is called compact (closed, $F_\sigma$, etc.) if the edge relation of $G$ is a compact (closed, $F_\sigma$, etc.) subset of $X^2$. A graph $G$ on $X$ is clopen (open) if the edge relation of $G$ is a clopen (open) subset of $X^2$ without the diagonal.    

Recall that a graph $G$ is \emph{$\omega$-saturated} if it satisfies condition (R) above, namely, for every disjoint finite sets $A, B \subs G$ there exists $v \in G \setminus (A \cup B)$ adjacent to all elements of $A$ and to no element of $B$.
A graph $G$ is \emph{$n$-saturated} if the above holds for sets $A, B$ with $|A \cup B| < n$.

\begin{lm}\label{LMirbiwrgbw}
	Let $K$ be a graph endowed with a compact Hausdorff topology such that the edge relation has closed sections.
	Let $D$ be a dense subset of $K$ and suppose that for every finite set $S \subs D$ there is a vertex in $K$ adjacent to all elements of $S$.
	Then $K$ contains a vertex adjacent to all other vertices.
\end{lm}

\begin{pf}
	Let $\Sigma = \fin D$.
	Given $S \in \Sigma$, choose $v_S \in K$ adjacent to all elements of $S$.
	Define
	$$F = \bigcap_{x \in D} \cl \setof{v_T}{T \in \Sigma,\; x \in T}.$$
	By compactness, $F \nnempty$.
	Choose $v \in F$ and suppose $v$ is not adjacent to some $u \in D$.
	Choose a neighborhood $V$ of $v$ such that $V \times \sn u$ is disjoint from the graph relation of $K$.
	There exists $T \in \Sigma$ such that $v_T \in V$ and $u \in T$.
	This gives a contradiction, showing that $v$ is adjacent to all elements of $D$.
	As $D$ is dense, $v$ is actually adjacent to each vertex of $K$.
\end{pf}


\begin{wn}\label{ThereAreNoOmegaSaturatedCompactGraphs}
	No compact graph is $\omega$-saturated. Moreover, no open graph on a compact space is $\omega$-saturated.
\end{wn}

\begin{pf}
	Suppose $K$ is an $\omega$-saturated compact graph.
	By Lemma~\ref{LMirbiwrgbw}, there is $v \in K$ adjacent to all other vertices.
	But now $\sn v$ witnesses the fact that $K$ is not $2$-saturated, a contradiction.
	
	Note that the dual graph to an  $\omega$-saturated graph is  $\omega$-saturated as well, and the dual graph to an open graph is compact. So the ``moreover'' part of the assertion follows from the already proved one.  
\end{pf}

This shows that no profinite graph (countable or uncountable) shares the properties of the classical random graph. There is also no compact topology on the random graph with closed edge relation. Note that for any compact metrizable topology on the random graph, the edge relation is $G_\delta$ and $F_\sigma$, so this observation cannot be strengthened. Note also that the first author showed in \cite{Geschke} that there is an $\omega_1$-saturated graph on $2^\omega$ with an $F_\sigma$ edge relation.  


Corollary \ref{ThereAreNoOmegaSaturatedCompactGraphs} shows that there are no $\omega$-saturated graphs on $2^\omega$. The first author proved in \cite{Geschke} that there is a clopen graph on $2^\omega$ which is $3$-saturated, while no clopen graph on $2^\omega$ can have infinite subgraphs that are $4$-saturated. The following natural question arises: Is there a closed graph on $2^\omega$ that is $n$-saturated for some finite $n >3$? This question is answered positively in the forthcoming~\cite{GG}.

\section{The projectively universal homogeneous graph}\label{SecPGebrgihbe}

Replacing embeddings by quotient maps, the general \fra\ theory gives us a profinite (called \emph{projective} by Irwin and Solecki~\cite{IrwSol}) graph $\PG$ that is the most complicated inverse limit of finite (nonempty) graphs. This graph and its relatives have been already studied by Camerlo~\cite{Camerlo}. We shall recall some of the results from~\cite{Camerlo} with slightly more direct arguments. We also present an explicit model of $\PG$.

From the general \fra\ theory~\cite{DrGoe92, K_flims} we infer that there is an inverse sequence
\[
\begin{tikzcd}
	G_0&\arrow[l, "q_{0}"'] G_1&\arrow[l, "q_{1}"']G_2&\arrow[l,"q_{2}"']\cdots
\end{tikzcd}
\]  
which has the opposite property to that leading to the random graph: for any finite graph $G$ and any quotient map $\map gG{G_k}$ there is a quotient map $\map f{G_n}G$ for some $n>k$ such that the diagram 
\[
\begin{tikzcd}
	G_k&    &\arrow[ll, "q^n_{k}"']\arrow[dl,dashed ,"f"']G_n\\
	&\arrow[ul,"g"']G& 
\end{tikzcd}
\]
commutes. Such a sequence is called a \emph{\fra\ sequence}, see~\cite{K_flims}. 

It turns out that if profinite graphs $H$ and $H'$ are inverse limits of \fra\ sequences, then there is a topological graph isomorphism $\map fH{H'}$, namely, a graph isomorphism that is at the same time a homeomorphism.
The following property characterizes $\PG$.

\begin{tw}\label{UniversalProjectiveGraph}
	There is a unique second countable profinite graph $\PG$ satisfying the condition
	\begin{enumerate}
		\item[$(\star)$] Given continuous quotient maps $\map f\PG S$ and $\map qTS$, where $S,T$ are nonempty finite graphs, there exists a continuous quotient map $\map g\PG T$ such that the diagram 
		\[
		\begin{tikzcd}
			S&    &\arrow[ll, "f"']\arrow[dl,dashed ,"g"']\PG\\
			&\arrow[ul,"q"']T& 
		\end{tikzcd}
		\]
		commutes.
	\end{enumerate}
	Furthermore, given any nonempty second countable profinite graph $K$, there exists a quotient map from $\PG$ to $K$. 
\end{tw}

Yet another ``furthermore'' part in the theorem above is: For every two continuous graph quotient maps $\map {p_i}\PG {F_i}$ ($i=0,1$) onto finite graphs, every isomorphism $\map h {F_0}{F_1}$ lifts to an automorphism $\map {\tilde h} \PG \PG$ in the sense that $h \cmp p_0 = p_1 \cmp {\tilde{h}}$. This is called \emph{projective homogeneity} and follows from $(\star)$ by a standard back-and-forth argument (again,~\cite{K_flims} contains details in the general setting of arbitrary categories).

We now discuss some properties of the graph $\PG$ that are relevant to our study.

\begin{prop}[Camerlo~\cite{Camerlo}]\label{PropCamerlo}
	Graph $\PG$ does not contain paths of length greater than one.
\end{prop}

\begin{pf}
	Suppose $a,b,c \in \PG$ is a path of length two, i.e., $\dn a b$ and $\dn b c$ are edges in $\PG$ and $a,b,c$ are pairwise distinct.
	Choose a decomposition of $\PG$ into three clopen sets $U_a, U_b$ and $U_c$, separating $\{a,b,c\}$. Let $S=\{U_a,U_b,U_c\}$. We define a graph relation $E_S$ on $S$ as follows: $UE_SU'$ if and only if there are $x\in U$ and $y\in U'$ such that $\{x,y\}$ is an edge in $\PG$. Define $\map f\PG S$ by $f(x)=U$ provided $x\in U$. Clearly $f$ is a quotient mapping. 
	
	Let $T$ be the graph obtained from $S$ by splitting the edge $U_b$ into two vertices $v_0, v_1$. Formally,
	$$T = (S \setminus \sn{U_b}) \cup \dn {v_0} {v_1}.$$
	The edges in graph $T$ are drawn in the following picture.
	
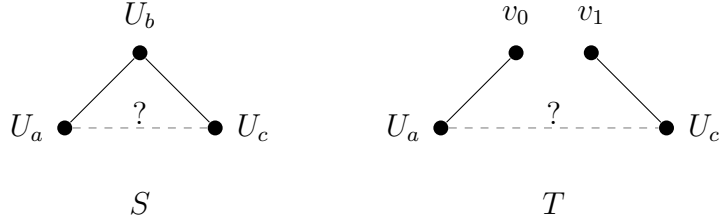
\begin{center}
		\begin{tikzpicture}[every node/.style={circle,inner sep=2pt}]
			\node[fill=black] (b) at (1,1) {};
			\node[fill=black] (a) at (0,0) {};
			\node[fill=black] (c) at (2,0) {};
			\draw (b) -- (a);
			\draw (b) -- (c);
			\draw[gray, dashed] (a) -- (c);
			
			\node at (-0.5,0) {$U_a$};
			\node at (2.5,0) {$U_c$};
			\node at (1,1.5) {$U_b$};
			\node at (1 , 0.2) {$?$};
			\node at (1,-1) {$S$};

			\node[fill=black] (v0) at (6,1) {};
			\node[fill=black] (v1) at (7,1) {};
			\node[fill=black] (a1) at (5,0) {};
			\node[fill=black] (c1) at (8,0) {};
			\draw (v0) -- (a1);
			\draw (v1) -- (c1);
			\draw[gray, dashed] (a1) -- (c1);
			
			\node at (4.5,0) {$U_a$};
			\node at (8.5,0) {$U_c$};
			\node at (6,1.5) {$v_0$};
			\node at (7,1.5) {$v_1$};  
			\node at (6.5 , 0.2) {$?$};
			\node at (6.5,-1) {$T$};
			
		\end{tikzpicture}
	\end{center}
	The dashed edge with question mark in graph $S$ means that there may be or not an edge between $U_a$ and $U_c$. The dashed edge in graph $T$ means that there is an edge between $U_a$ and $U_c$ provided it was in $S$. Moreover, there are no more edges in $T$ than the drawn ones. 
	
	Let $\map g T S$ be the canonical surjection, i.e. $g$ is the identity on $S \setminus \sn{U_b}$ and $g(v_i) = U_b$ for $i=0,1$. From the shape of $T$ we immediately obtain that $g$ a quotient mapping of graphs.
	
	By Theorem \ref{UniversalProjectiveGraph}, we can find a continuous quotient map $\map q \PG T$ such that $g \cmp q = f$. Suppose $q(b) = v_0$ (the other case can be handled in the same way). Then $g(q(b)) = f(b)$ and also $g(q(c)) = f(c)$.
	On the other hand, $q$ is a graph homomorphism and $b$, $c$ are adjacent in $\PG$, therefore the same must hold for $q(b)=v_0$ and $q(c)=f(c)$ in $T$.
	This is a contradiction.
\end{pf}

We now see that the edge relation in $\PG$ is actually an equivalence relation, whose equivalence classes contain at most two elements. Irwin and Solecki~\cite{IrwSol} constructed the projective \fra\ limit of finite linear graphs; it turned out that the graph relation on the limit, as in our case, is an equivalence relation, whose equivalence classes contain at most two elements. Then they showed that its quotient with respect to the edge relation is a chainable hereditarily indecomposable continuum. Inspired by \cite{IrwSol}, Barto\v{s}ov\'{a} and Kwiatkowska~\cite{BarKwiat} constructed the projective \fra\ limit of a special kind of finite trees whose graph relation was an equivalence relation; its quotient with respect to the edge relation was shown to be the so-called \emph{Lelek fan}---a rather intriguing planar continuum.

Unfortunately, the quotient of $\PG$ with respect to its edge relation is again homeomorphic to the Cantor set---already proved by Camerlo~\cite{Camerlo}. We provide the arguments just for the sake of completeness and in order to make this note more self-contained. Some of these arguments will also be used later.
It is worth noting that Camerlo~\cite{Camerlo} provided a full topological characterization of all possible quotients $X/R$ of inverse \fra\ limits $(X,R)$, where $R$ is an equivalence relation.   

Note that $\PG$ is homeomorphic to the Cantor set. Indeed, condition $(\star)$ says that splitting every point into two points is realized in $\PG$, therefore $\PG$ has no isolated points.


\begin{prop}[{Camerlo~\cite{Camerlo}}]\label{TheoremCamerlo}
	The edge relation $E$ on $\PG$ is a closed equivalence relation on $2^{\omega}$.
	The quotient $2^{\omega}/E$ is homeomorphic to $2^{\omega}$.
\end{prop}

\begin{proof}
	We show that any two elements of $2^{\omega}/E$ can be separated by clopen sets.
	Let $x,y\in 2^{\omega}$ be two vertices of $\PG$ that are not equivalent, i.e., assume that $x$ and $y$ are distinct and not connected by an edge.
	Since the edge-relation of $\PG$ is closed, there are disjoint clopen neighborhoods $U$ and $V$ of $x$ and $y$, respectively, such that no edge of $\PG$ runs from $U$ to $V$.  
	Shrinking $U$ and $V$ if necessary, we may assume that $W=2^{\omega}\setminus(U\cup V)$ is nonempty.
	
	Let $A$ be a graph with three vertices $u$, $v$, and $w$ such that $\map f\PG A$ mapping all of $U$ to $u$, all of $V$ to $v$ and all of $W$ to $w$, is a strict homomorphism.
	By the choice of $U$ and $V$, $A$ has no edge from $u$ to $v$. 
	Let $B$ be obtained from $A$ by splitting $w$ into two vertices $w_u$ and $w_v$.
	The edges of $B$ are defined as follows:  there is an edge from $w_u$ to $u$ iff in $A$ there is an edge from $w$ to $u$;
	there is an edge from $w_v$ to $v$ iff in $A$ there is an edge from $w$ to $v$.
	$B$ has no further edge.
	Let $\map gBA$ be the strict homomorphism mapping $u$ to $u$, $v$ to $v$ and both $w_u$ and $w_v$ to $w$.
	
	By the universal property of $\PG$, there is a continuous strict homomorphism $\map h\PG B$ such that 
	$f=g\circ h$.
	Now $U'=h^{-1}[\{u,w_u\}]$ and $V'=h^{-1}[\{v,w_v\}]$ are disjoint clopen subsets of $2^{\omega}$ such that no edge of 
	$\PG$ runs from $U'$ to $V'$.  I.e., the sets $U'/E$ and $V'/E$ are disjoint.
	Moreover, $2^{\omega}=(U'/E)\cup(V'/E)$.  It follows that $U'/E$ and $V'/E$ are disjoint clopen neighborhoods of 
	the equivalence classes $[x]_E$ and $[y]_E$.
	
	This shows that $2^{\omega}/E$ is zero-dimensional.  
	It remains to show that it has no isolated points.  
	But an equivalence class $[x]_E$ is an isolated point in the quotient $2^{\omega}/\sim$ only if it is open as a subset of $2^{\omega}$.  Since no equivalence class has more than two elements, this never happens.
\end{proof}

The next result shows that all isolated vertices, namely, vertices inducing one-element equivalence classes with respect to $E$, form a topologically large set in $\PG$. 

\begin{tw}\label{THMhuwvwe}
	Graph $\PG$ has a dense $G_\delta$ set of isolated vertices.
\end{tw}

\begin{proof}
	It suffices to show that condition (IV) of Lemma~\ref{THMKriterionIsoVerticess} is satisfied. So fix a nonempty clopen set $U \subs \PG$. Let $\map f \PG A$ be a quotient homomorphism onto a two-element graph, such that $U = f^{-1}(a_0)$, where $A = \dn{a_0}{a_1}$. Let $B$ be the graph obtained from $A$ by adding an isolated vertex $b$ and let $\map qBA$ be a quotient homomorphism mapping $b$ to $a_0$, identity on $A$. Condition ($\star$) of Theorem~\ref{UniversalProjectiveGraph} gives a quotient map $\map g \PG B$ such that $q \cmp g = f$. Let $W = g^{-1}(b)$. Then $W \subs U$ is a nonempty clopen set.
	
	Suppose there are adjacent vertices $w,v$ with $w \in W$ and $v \in \PG \setminus W$.
	Then $b = g(w)$ is adjacent to $g(v) \in \dn{a_0}{a_1}$, a contradiction. This shows that there are no adjacencies between $W$ and $\PG \setminus W \sups \PG \setminus U$. Thus (IV) holds true.
\end{proof}

We now see that constructing a ``nice" universal compact graph by using all quotient maps is not the best idea.
We have proposed to refine our tools in order to obtain a better universal compact graph, namely, using retractions.
On the other hand, we believe that the graph $\PG$ still deserves some attention, in particular, the topology of its edge relation is far from being trivial. Below we present its concrete model.

\subsection{A concrete model of $\PG$}
\label{SectionModelQuotient}

Below we present a model of the projectively universal profinite graph $\PG$. Consider the following graph. Let $G=(\{0,1,2,3\}^\omega,E)$ be a graph where $xEx$ for every $x\in G$, and for $x\neq y$ we put $xEy$ whenever
$$
\Big(\exists n\Bigl(\forall k<n(x(k)=y(k))\wedge (\{x(n),y(n)\}=\{2,3\})\wedge\forall k>n (x(k)=y(k)\in\{0,1\})\Bigr)
$$
$$
\text{ or }\{x(0),y(0)\}=\{0,1\}\wedge\forall k\geq 1 (x(k)=y(k)\in\{0,1\})\Big).
$$
We will show that $G$ is the inverse limit of a \fra\ sequence of finite graphs with quotients. 

Let us construct the sequence $(G_n)$ of finite graphs. Let $G_n=\{0,1,2,3\}^n$ and let $E_n$ be a relation on $G_n$ such that $xE_nx$ for every $x\in G_n$, and for $x\neq y$, we put $xE_ny$ whenever
$$
\Big(\exists l<n\Big(\forall k<l(x(k)=y(k))\wedge \{x(l),y(l)\}=\{2,3\}\wedge\forall k>l (x(k)=y(k)\in\{0,1\})\Big)
$$
$$
\text{ or }\{x(0),y(0)\}=\{0,1\}\wedge\forall k\geq 1 (x(k)=y(k)\in\{0,1\})\Big).
$$
We define $\map {p_{n}}{G_{n+1}}{G_{n}}$ by $p_{n}(x)=x\vert_{n}$, where $x\vert_{n}$ is just a short way of writing $(x(0),x(1),\dots,x(n-1))$. Let us observe the following.

\begin{lm}
	$p_{n}$ is a quotient map.
\end{lm}

\begin{proof}
	Firstly, we prove that $p_{n}$ is a graph homomorphism. Assume that $x E_{n+1} y$ and $x\neq y$. There are three cases we need to consider.\\
	1. $x(n)=2$ and $y(n)=3$ and $x(k)=y(k)$ for $k<n$. Then $x\vert_{n}=y\vert_{n}$, and therefore $p_{n}(x)=p_{n}(y)$ and $p_{n}(x)E_np_{n}(y)$.\\
	2. There is $l<n$ with $x(l)=2$, $y(l)=3$ and $x(k)=y(k)$ for $k<l$, and $x(k)=y(k)\in\{0,1\}$ for $k>l$. Then $p_{n}(x)\neq p_{n}(y)$ and $p_{n}(x)E_{n}p_{n}(y)$.\\
	3. $\{x(0),y(0)\}=\{0,1\}$ and $(x(k)=y(k)\in\{0,1\})$ for $k>0$. Then as before we obtain $p_{n}(x)\neq p_{n}(y)$ and $p_{n}(x)E_{n}p_{n}(y)$ immediately from the definition of $E_n$.
	
	For $a=(a(0),a(1),\dots,a(n-1))$ we define the concatenation $a\concat0$ of $a$ and 0 as $(a(0),a(1),\dots,a(n-1),0)$. Let $a,b$ be adjacent in $G_{n}$. Note that $a\concat0E_{n+1}b\concat0$, $p_{n}(a\concat0)=a$ and $p_{n}(b\concat0)=b$. This shows that $p_{n}$ is a strict graph homomorphism. Since $p_{n}$ is onto $G_{n}$, it is a quotient mapping. 
\end{proof}

If $n>k$ then $p^n_{k}=p_{k}\circ p_{k+1}\circ\dots\circ p_{n-1}$ is the restriction to the first $k$ coordinates. Clearly $p^n_{k}\circ p_n=p_k$. The following picture shows what the graphs $G_1$ and $G_2$ look like.\\ 
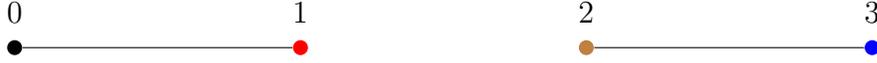
\vspace{1cm}
\begin{tikzpicture}[scale = 0.95, every node/.style={circle,inner sep=2pt}]
	\node[fill=black] (A) at (0,0) {};
	\node[fill=red] (B) at (4,0) {};
	\node[fill=brown] (C) at (8,0) {};
	\node[fill=blue] (D) at (12,0) {};
	
	\draw (A) -- (B);
	\draw (C) -- (D);
	
	\node at (0,0.5) {$0$};
	\node at (4,0.5) {$1$};
	\node at (8,0.5) {$2$};
	\node at (12,0.5) {$3$};
	
\end{tikzpicture}
\newline 
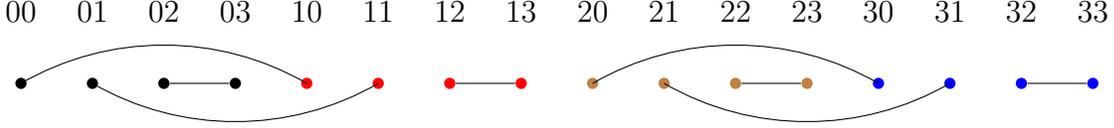
\begin{tikzpicture}[scale = 0.95, every node/.style={circle,inner sep=1.5pt}]
	\node[fill=black] (A1) at (0,0) {};
	\node[fill=black] (A2) at (1,0) {};
	\node[fill=black] (A3) at (2,0) {};
	\node[fill=black] (A4) at (3,0) {};
	\node[fill=red] (B1) at (4,0) {};
	\node[fill=red] (B2) at (5,0) {};
	\node[fill=red] (B3) at (6,0) {};
	\node[fill=red] (B4) at (7,0) {};
	\node[fill=brown] (C1) at (8,0) {};
	\node[fill=brown] (C2) at (9,0) {};
	\node[fill=brown] (C3) at (10,0) {};
	\node[fill=brown] (C4) at (11,0) {};
	\node[fill=blue] (D1) at (12,0) {};
	\node[fill=blue] (D2) at (13,0) {};
	\node[fill=blue] (D3) at (14,0) {};
	\node[fill=blue] (D4) at (15,0) {};
	
	\draw (A1) arc [radius=4, start angle=120, end angle= 60] ;
	\draw (A2) arc [radius=4, start angle=-120, end angle=-60] ;  
	\draw (A3) -- (A4);
	\draw (B3) -- (B4);
	\draw (C3) -- (C4);
	\draw (D3) -- (D4);
	\draw (C1) arc [radius=4, start angle=120, end angle= 60] ;
	\draw (C2) arc [radius=4, start angle=-120, end angle=-60] ;  
	
	\node at (0,1) {$00$};
	\node at (1,1) {$01$};
	\node at (2,1) {$02$};
	\node at (3,1) {$03$};
	\node at (4,1) {$10$};
	\node at (5,1) {$11$};
	\node at (6,1) {$12$};
	\node at (7,1) {$13$};
	\node at (8,1) {$20$};
	\node at (9,1) {$21$};
	\node at (10,1) {$22$};
	\node at (11,1) {$23$};
	\node at (12,1) {$30$};
	\node at (13,1) {$31$};
	\node at (14,1) {$32$};
	\node at (15,1) {$33$};
\end{tikzpicture}

The quotient map $p_{2,1}$ maps the vertices of one color in $G_2$ to the one vertex of that color in $G_1$. Below we present $G_3$  which gives a flavor  of our construction. \\
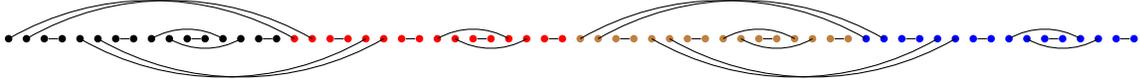
\begin{tikzpicture}[scale = 0.95, every node/.style={circle,inner sep=1pt,fill=black}]
	\node (A11) at (0,0) {};
	\node (A21) at (0.25,0) {};
	\node (A31) at (0.5,0) {};
	\node (A41) at (0.75,0) {};
	\node (A12) at (1,0) {};
	\node (A22) at (1.25,0) {};
	\node (A32) at (1.5,0) {};
	\node (A42) at (1.75,0) {};
	\node (A13) at (2,0) {};
	\node (A23) at (2.25,0) {};
	\node (A33) at (2.5,0) {};
	\node (A43) at (2.75,0) {};
	\node (A14) at (3,0) {};
	\node (A24) at (3.25,0) {};
	\node (A34) at (3.5,0) {};
	\node (A44) at (3.75,0) {};
	
	\node (B11)[color=red] at (4,0) {};
	\node (B21)[color=red] at (4.25,0) {};
	\node (B31)[color=red] at (4.5,0) {};
	\node (B41)[color=red] at (4.75,0) {};
	\node (B12)[color=red] at (5,0) {};
	\node (B22)[color=red] at (5.25,0) {};
	\node (B32)[color=red] at (5.5,0) {};
	\node (B42)[color=red] at (5.75,0) {};
	\node (B13)[color=red] at (6,0) {};
	\node (B23)[color=red] at (6.25,0) {};
	\node (B33)[color=red] at (6.5,0) {};
	\node (B43)[color=red] at (6.75,0) {};
	\node (B14)[color=red] at (7,0) {};
	\node (B24)[color=red] at (7.25,0) {};
	\node (B34)[color=red] at (7.5,0) {};
	\node (B44)[color=red] at (7.75,0) {};
	
	\node (C11)[color=brown] at (8,0) {};
	\node (C21)[color=brown] at (8.25,0) {};
	\node (C31)[color=brown] at (8.5,0) {};
	\node (C41)[color=brown] at (8.75,0) {};
	\node (C12)[color=brown] at (9,0) {};
	\node (C22)[color=brown] at (9.25,0) {};
	\node (C32)[color=brown] at (9.5,0) {};
	\node (C42)[color=brown] at (9.75,0) {};
	\node (C13)[color=brown] at (10,0) {};
	\node (C23)[color=brown] at (10.25,0) {};
	\node (C33)[color=brown] at (10.5,0) {};
	\node (C43)[color=brown] at (10.75,0) {};
	\node (C14)[color=brown] at (11,0) {};
	\node (C24)[color=brown] at (11.25,0) {};
	\node (C34)[color=brown] at (11.5,0) {};
	\node (C44)[color=brown] at (11.75,0) {};
	
	\node (D11)[color=blue] at (12,0) {};
	\node (D21)[color=blue] at (12.25,0) {};
	\node (D31)[color=blue] at (12.5,0) {};
	\node (D41)[color=blue] at (12.75,0) {};
	\node (D12)[color=blue] at (13,0) {};
	\node (D22)[color=blue] at (13.25,0) {};
	\node (D32)[color=blue] at (13.5,0) {};
	\node (D42)[color=blue] at (13.75,0) {};
	\node (D13)[color=blue] at (14,0) {};
	\node (D23)[color=blue] at (14.25,0) {};
	\node (D33)[color=blue] at (14.5,0) {};
	\node (D43)[color=blue] at (14.75,0) {};
	\node (D14)[color=blue] at (15,0) {};
	\node (D24)[color=blue] at (15.25,0) {};
	\node (D34)[color=blue] at (15.5,0) {};
	\node (D44)[color=blue] at (15.75,0) {};
	
	\draw (A11) arc [radius=4, start angle=120, end angle= 60] ;  
	\draw (A21) arc [radius=4, start angle=120, end angle=60] ;  
	\draw (A12) arc [radius=4, start angle=-120, end angle= -60] ;  
	\draw (A22) arc [radius=4, start angle=-120, end angle=-60] ;  
	\draw (A31) -- (A41);
	\draw (A32) -- (A42);
	\draw (A33) -- (A43);
	\draw (A34) -- (A44);
	\draw (A13) arc [radius=1, start angle=120, end angle=60] ;
	\draw (A23) arc [radius=1, start angle=-120, end angle=-60] ;
	
	\draw (B31) -- (B41);
	\draw (B32) -- (B42);
	\draw (B33) -- (B43);
	\draw (B34) -- (B44);
	\draw (B13) arc [radius=1, start angle=120, end angle= 60] ;
	\draw (B23) arc [radius=1, start angle=-120, end angle=-60] ;
	
	\draw (C11) arc [radius=4, start angle=120, end angle= 60] ;  
	\draw (C21) arc [radius=4, start angle=120, end angle=60] ;  
	\draw (C12) arc [radius=4, start angle=-120, end angle= -60] ;  
	\draw (C22) arc [radius=4, start angle=-120, end angle=-60] ;  
	\draw (C31) -- (C41);
	\draw (C32) -- (C42);
	\draw (C33) -- (C43);
	\draw (C34) -- (C44);
	\draw (C13) arc [radius=1, start angle=120, end angle= 60] ;
	\draw (C23) arc [radius=1, start angle=-120, end angle=-60] ;
	
	\draw (D31) -- (D41);
	\draw (D32) -- (D42);
	\draw (D33) -- (D43);
	\draw (D34) -- (D44);
	\draw (D13) arc [radius=1, start angle=120, end angle= 60] ;
	\draw (D23) arc [radius=1, start angle=-120, end angle=-60] ;
	
\end{tikzpicture}

Consider the sequence of graphs
\begin{equation}\label{FraSeq}
	\begin{tikzcd}
		G_0&\arrow[l, "p^1_{0}"'] G_1&\arrow[l, "p^2_{1}"']G_2&\arrow[l,]\cdots
	\end{tikzcd}
\end{equation}
to show that $G$ is isomorphic to its inverse limit. We need to show the following. 
\begin{lm}
	$xEy$ if and only if $p_{n}(x)E_np_{n}(y)$ for every $n$.
\end{lm}
\begin{proof}
	If $xEy$, then clearly $x\vert_n E_n y\vert_n$ for every $n$. 
	
	Assume now, that $p_{n}(x)E_np_{n}(y)$ for all $n$. If $x=y$, then we are done. Otherwise there is the smallest $l$ with $x(l)\neq y(l)$. Since $x\vert_{l+1} E_{l+1}y\vert_{l+1}$, then $\{x(l),y(l)\}=\{2,3\}$ and $x(k)=y(k)$ for $k<l$. If $l=0$, then $\{x(l),y(l)\}=\{0,1\}$ or $\{x(l),y(l)\}=\{2,3\}$. Using the fact that $p_{k+1}(x)E_{k+1}p_{k+1}(y)$ for $k>l$ we obtain that $x(k)=y(k)\in\{0,1\}$ for $k>l$. Hence $xEy$. 
\end{proof}

\begin{prop}
	The sequence \eqref{FraSeq} is a \fra\ sequence in the category of finite graphs with quotient maps. 
\end{prop}

\begin{proof}
	Let $H$ be a finite graph and let $\map pHG_k$ be a quotient map. Note that each vertex $x\in G_k$ is adjacent to exactly one vertex $y\in G_k$, $y\neq x$. Since $\vert G_k\vert = 4^k$, then $G_k$ has $2\cdot 4^{k-1}$ edges. Let $(D_i)_{i=1}^{2\cdot 4^{k-1}}$ be the set of all edges in $G_k$. 
	
	Fix an edge $\{x,y\}$ in $G_k$. Then for any $m>k$
	\begin{itemize}
		\item there are at least $2^{m-k-1}$ edges $\{u,v\}$ in $G_m$ with $p^m_k(u)=p^m_k(v)=x$; they are of the form $u=x\concat 3 \concat t$ and $v=x \concat 2 \concat t$ where $t$ is a $0$--$1$ sequence of length $m-k-1$;
		\item there are $2^{m-k}$ edges $\{u,v\}$ in $G_m$ with $p^m_k(u)=x$ and $p^m_k(v)=y$; they are of the form $u=x \concat t$ and $v=y \concat t$ where $t$ is a $0$--$1$ sequence of length $m-k$.
	\end{itemize} 
	Let us fix $m$ such that $2^{m-k-1}$ is strictly greater than the number of vertices and edges of $H$.  
	
	For every $i\leq 2\cdot 4^{k-1}$ we will find a quotient map $\map {g_i}{(p^m_{k})^{-1}(D_i)}{p^{-1}(D_i)}$ such that the diagram
	\begin{equation}\label{diag1}
		\begin{tikzcd}
			D_i &&\arrow[ll, "p^m_{k}"']  (p^m_{k})^{-1}(D_i) \arrow[ld,dashed,"g_i"] \\
			&p^{-1}(D_i) \arrow[lu,"p"]&
		\end{tikzcd}
	\end{equation}
	commutes.

	Suppose that we have already found such $g_i$'s. We define $\map g{G_m}H$ as follows: $g(x)=g_i(x)$ provided $x\in (p^m_{k})^{-1}(D_i)$. Since there are no edges between $D_i$ and $D_j$ for $i\neq j$, $p$ and $p^m_{k}$ are quotients, there are no edges between $(p^m_{k})^{-1}(D_i)$ and $(p^m_{k})^{-1}(D_j)$, and no edges between $p^{-1}(D_i)$ and $p^{-1}(D_j)$. From this we easily obtain that $g$ is a quotient map. Clearly, the diagram
	\[
	\begin{tikzcd}
		G_k &&\arrow[ll, "p^m_{k}"']  G_m \arrow[ld,"g"] \\
		&H \arrow[lu,"p"]&
	\end{tikzcd}
	\]
	commutes.

	Now, it is enough to construct $g_i$'s. Let as color $D_i$ vertices $D_i=\{{\rm black},{\rm red}\}$. We say that a vertex $x\in H$ is black (or red) if $p(x)$ is black (or red). Isolated verices of $H$ are divided into black isolated $BI$ and read isolated $RI$. The edges $\{x,y\}$, with $x\neq y$, in $H$ are divided into black edges $BE$ if $p(x)$ and $p(y)$ are both black, red edges $RE$ if $p(x)$ and $p(y)$ are both red, and  black-and-red edges $BRE$ if $p(x)$ and $p(y)$ have different colors. Since $p$ is quotient map, then $BRE\neq\emptyset$. Let us fix $\{r,b\}\in BRE$. Put $\ell_1:=\vert BI\vert$, $\ell_2:=\vert BE\vert$, $\ell_3:=\vert RI\vert$, $\ell_4:=\vert RE\vert$, and $\ell_5:=\vert BRE\vert$. Then $\ell_1,\ell_2,\ell_3,\ell_4\geq 0$ and $\ell_5\geq 1$.

	Recall that $(p^m_{k})^{-1}(D_i)$ does not contain isolated vertices. We divide its edges $\{x,y\}$ into three categories: black $\mathbf{B}$ if $p^m_{k}$ maps $x$ and $y$ into black, red $\mathbf{R}$ if $p^m_{k}$ maps $x$ and $y$ into red, and black-and-red $\mathbf{BR}$ if $p^m_{k}$ maps $x$ and $y$ into two colors. Divide $\mathbf{B}$ into three pairwise disjoint sets $\mathbf{B}_1,\mathbf{B}_2$ and $\mathbf{B}_3$ with $\vert\mathbf{B}_1\vert=\ell_1$, $\vert\mathbf{B}_2\vert=\ell_2$ and $\mathbf{B}_1\cup \mathbf{B}_2\cup \mathbf{B}_3=\mathbf{B}$. This is possible, since $m$ is so large that $(p^m_{k})^{-1}(D_i)$ contains more black edges than $H$ has all edges. Similarly, divide $\mathbf{R}$ into three pairwise disjoint sets $\mathbf{R}_1,\mathbf{R}_2$ and $\mathbf{R}_3$ with $\vert\mathbf{R}_1\vert=\ell_3$, $\vert\mathbf{R}_2\vert=\ell_4$ and $\mathbf{R}_1\cup \mathbf{R}_2\cup \mathbf{R}_3=\mathbf{R}$. Finally, divide $\mathbf{BR}$ into two disjoint sets $\mathbf{BR}_1$ and $\mathbf{BR}_2$ with $\vert\mathbf{BR}_1\vert=\ell_5$. 
	
	Now we are ready to define $g_i$:
	\begin{itemize}
		\item if $\{x,y\}$ is the $j$-th edge in $\mathbf{B}_1$, then $g_i(x)=g_i(y)$ is the $j$-th element of $BI$;
		\item if $\{x,y\}$ is the $j$-th edge in $\mathbf{B}_2$, then $\{g_i(x),g_i(y)\}$ is the $j$-th edge of $BE$;
		\item if $\{x,y\}$ is an edge in $\mathbf{B}_3$, then $g_i(x)=g_i(y)=b$;
		\item if $\{x,y\}$ is the $j$-th edge in $\mathbf{R}_1$, then $g_i(x)=g_i(y)$ is the $j$-th element of $RI$;
		\item if $\{x,y\}$ is the $j$-th edge in $\mathbf{R}_2$, then $\{g_i(x),g_i(y)\}$ is the $j$-th edge of $RE$;
		\item if $\{x,y\}$ is an edge in $\mathbf{R}_3$, then $g_i(x)=g_i(y)=r$;
		\item if $\{x,y\}$ is the $j$-th edge in $\mathbf{BR}_1$, then $\{g_i(x),g_i(y)\}$ is the $j$-th edge of $BRE$ such that $x$ and $g_i(x)$ have the same color, and $y$ and $g_i(y)$ have the same color; 
		\item if $\{x,y\}$ is an edge in $\mathbf{BR}_2$, then $\{g_i(x),g_i(y)\}=\{b,r\}$ such that $x$ and $g_i(x)$ have the same color, and $y$ and $g_i(y)$ have the same color.
	\end{itemize}          
	
	By the construction, $g_i$ is the quotient map and diagram \eqref{diag1} commutes.  
\end{proof}

\section{Retractable structures}

We now make an interlude with a few general facts concerning retractions and right-invertible mappings. Recall that a homomorphism $f$ is \emph{right-invertible} if there exists a homomorphism $i$ such that $f \cmp i$ is the identity. In that case, $i$ is necessarily an embedding and $f$ is a quotient mapping. Left-invertibility is defined in a symmetric way. A \emph{retraction} is a self-homomorphism that is the identity on its range. Note that, given a left-invertible homomorphism $i$ and given its left inverse $f$, we have a retraction $f \cmp i$.
Conversely, given a retraction $\map r X X$ and taking $f$ to be the mapping $r$ treated as a surjection onto $\img rX$, we obtain a right-invertible homomorphism $f$ whose left inverse is the inclusion $i \colon \img rX \subs X$. Of course, all these concepts are purely category-theoretic and they received a significant attention in general and algebraic topology, Banach space theory (where retractions are called \emph{bounded linear projections}), and in theoretical computer science, particularly in domain theory.

In what follows, we shall often consider retractions as surjections $\map f XY$ with $Y \subs X$, where the inclusion $Y \subs X$ is a natural choice for a right inverse to $f$.

\separator

Let us assume that we are given a class of relational structures (i.e. graphs, partially ordered sets, etc.). Our goal now is to select two subclasses related to retractions. We start with a general fact.

\begin{prop}\label{PROPerbigowhrgoiw}
	Let $X$ be a set and assume $\map f XX$ is a mapping such that $\img f X$ is finite. Then there exists $m>0$ such that $f^m$ is a retraction, that is, $f^m \cmp f^m = f^m$.
\end{prop}

\begin{pf}
	The semigroup $S = \setof{f^n}{n \in \Nat^+}$ is finite, because $f^{n+1} = f^n \cmp f$, where $f^n$ acts on the image of $f$, so there are only finitely many possibilities. Thus, there are $k,r>0$ such that $f^{k} = f^{k+r}$. An obvious induction shows that $f^n= f^{n+r}$ for every $n \goe k$. Finally,
	$f^{k r} \cmp f^{k r} = f^{kr}$,
	therefore $f^{k r}$ is a retraction.
\end{pf}

A well-known (and not completely trivial) fact is that every finite semigroup has an idempotent, however, in the proof above the semigroup is generated by one element, therefore finding an idempotent is very easy.

A (relational) structure $X$ is \emph{finitely retractable} if for every finite $F \subs X$ there is a retraction $\map r X X$ with $F \subs \img r X$ and $\img r X$ finite.

A structure $X$ is \emph{point-by-point retractable} (briefly: \emph{PPR}) if $X = \bigcup_{\ntr} X_n$, where $|X_0|=1$, $X_n \subs X_{n+1}$, $|X_{n+1} \setminus X_n| \loe 1$ and $X_n$ is a retract of $X_{n+1}$ for each $\ntr$. Clearly, a point-by-point retractable structure is necessarily countable.

These two definitions seem to be a bit unrelated. The following standard claim clarifies the situation. A more general, purely category-theoretic, statement can be found in~\cite[Section 6]{K_flims}.

\begin{prop}\label{PROPSztyryDwa}
	A countable structure $X$ is finitely retractable if and only if there is a sequence of retractions $\sett{r_n}{\ntr}$ of $X$ such that $r_n \cmp r_m = r_{\min(n,m)}$ for every $n,m \in \nat$ and $\img {r_n}X$ is finite for every $\ntr$.
	
	A structure $X$ is point-by-point retractable if and only if there exists a sequence $\sett{r_n}{\ntr}$ as above, moreover satisfying $|\img {r_n}X| \loe n$ for every $\ntr$.
\end{prop}

\begin{pf}
	Assume $X$ is finitely retractable.
	Write $X = \bigcup_{\ntr} X_n$, where each $X_n$ is a finite retract of $X$ and $X_n \subs X_{n+1}$ for every $n$. This can be done by a straightforward induction.
	
	Denote by $r_n^{n+1}$ a fixed retraction from $X_{n+1}$ onto $X_n$. Given $m>n$, define
	$$r^m_n = r^{n+1}_n \cmp \dots \cmp r^m_{m-1}.$$
	Then $r^m_n$ is a retraction from $X_m$ onto $X_n$. Given $x \in X$, define
	$$r_n(x) = r^m_n(x),$$
	where $m>n$ is such that $x \in X_m$. This does not depend on the choice of $m$ and defines a retraction from $X$ onto $X_n$.
	Finally, $r_n \cmp r_m = r_{\min(n,m)}$ for every $n,m$.
	
	The same argument shows the second part, now just assuming that $X_n$ is a retract of $X_{n+1}$ for each $n$.
\end{pf}

Below we provide another characterization of countable finitely retractable structures.

\begin{tw}\label{THMboundedApproxims}
	Assume $\sett{\map{f_n}{X}{X}}{\ntr}$ is a sequence of self-homomorphisms of a fixed relational structure $X$ such that $\img{f_n}{X}$ is finite for each $\ntr$ and for every $x \in X$ there is $n_0$ such that $f_n(x) = x$ for $n\goe n_0$ (in other words, $f_n$ converges pointwise to the identity, when $X$ is endowed with the discrete topology).
	Then $X$ is countable and finitely retractable.
\end{tw}

This is in fact a characterization, due to Proposition~\ref{PROPSztyryDwa}. 

\begin{pf}
	Clearly, $X = \bigcup_{\ntr}\img {f_n}X$ is countable.
	Fix a finite substructure $F \subs X$ and fix $m$ such that $f_m(x)=x$ for every $x \in F$. By Proposition~\ref{PROPerbigowhrgoiw}, there is $k>0$ such that $f_m^k$ is a retraction. Obviously, $F \subs \img{f_m^k}{X}$.	
\end{pf}

\begin{tw}\label{THMretsFinrectsha}
	Every retract of a finitely retractable structure is finitely retractable.
\end{tw}

\begin{pf}
	Assume $X$ is finitely retractable and $\map s X X$ is a retraction onto $Y \subs X$.
	Fix a finite substructure $F \subs Y$ and find a finite retract $A$ of $X$ with $F \subs A$. Let $\map r X X$ be a retraction onto $A$.
	Consider $f = s \cmp r$. Note that $f$ is identity on $F$ and its image is finite. Thus, by Proposition~\ref{PROPerbigowhrgoiw}, there is $k>0$ such that $f^k$ is a retraction. Its image is a finite retract of $Y$ containing $F$.	
\end{pf}

Note that every finite structure is obviously finitely retractable while finite structures (e.g. graphs) might not be point-by-point retractable: For instance, the 5-element cycle does not admit any retraction on its 4-element subgraph. 

We do not know whether retracts of PPR structures are PPR, however we have a weaker version of Theorem~\ref{THMretsFinrectsha}.

\begin{tw}\label{THMRetraktyFinitoPoPaPr}
	A finite retract of a PPR structure is PPR.
\end{tw}

\begin{pf}
	We use induction on the cardinality of a structure.
	Fix a finite structure $A$ which is a retract of some PPR structure and assume the statement above is true for structures of cardinality $<|A|$. Fix a minimal with respect to cardinality PPR structure $B \sups A$ together with a retraction $\map r BB$ whose image is $A$. Then $B = B' \cup \sn v$ and there is a retraction $\map s B B$ onto $B'$. By minimality, $v \in A$. Let $A' = A \cap B' = A \setminus \sn v$. We claim that $r(s(v)) \in A'$.
	
	Suppose otherwise, that is $r(s(v)) = v$. Let $C = A' \cup \sn{s(v)}$.
	Then $s$ is a homomorphism from $A$ onto $C$ whose inverse is $r$ restricted to $C$. Thus $A$ is isomorphic to $C$ and $s \cmp r$ is a retraction of $B'$ onto $C$. On the other hand, $B'$ is a smaller PPR structure, a contradiction.
	
	We have shown that $r \cmp s$ is a retraction of $A$ onto $A'$. Moreover, $A'$ is a retract of a PPR structure (namely, $B'$) therefore by the inductive hypothesis it is PPR.
	Finally, $A$ is PPR.
\end{pf}

\subsection{Envelopes}

We now come back to profinite structures, namely, inverse limits of finite ones.
Namely, given a finitely retractable structure $X$ together with a witnessing chain of finite retracts $\ciag{X}$, we actually obtain a natural profinite structure containing $X$ as a dense substructure. Below we describe the general idea.

Formally, for each $n$ let us fix a retraction $\map{r^{n+1}_n}{X_{n+1}}{X_n}$ and for every $m>n$ let $r^m_n$ be the suitable composition forming a retraction from $X_m$ onto $X_n$.
By this way, we turn $\ciag X$ into an inverse sequence and therefore we can look at its limit $\ovr X$ as a profinite structure.
General topological considerations (see~\cite[Section~3]{KuMi}) give the following

\begin{prop}\label{PROPdensitySixkm}
	$X$ is a dense subset of $\ovr X$.
\end{prop}

We shall call $\ovr X$ an \emph{envelope} of $X$. The envelope of course depends on the choice of retractions $r^{n+1}_n$.
We refer to~\cite[Section~6]{K_flims} for details in the more general setting, namely, categories of embedding-projection pairs.

\section{The universal homogeneous finitely retractable graph and its envelope}

Consider the category $\catligraphs$ of all nonempty finite graphs with left-invertible embeddings.
It is easy to see that all the axioms of \fra\ theory are satisfied. Namely,
$\catligraphs$ has the amalgamation property, as the standard amalgamation (not adding any extra edges) provides left-invertible embeddings. Directedness follows from the amalgamation property, as the graph with a single vertex is a retract of every nonempty graph. Finally, the category has only countably many isomorphic types and all hom-sets are finite.

Let $\ciagi{\catligraphs}$ be the category of all colimits of sequences in $\catligraphs$. By Proposition~\ref{PROPSztyryDwa}, its objects are precisely the countable finitely retractable graphs. Let $\ufr$ denote the \fra\ limit of $\catligraphs$, which exists by the general \fra\ theory~\cite{DrGoe92, K_flims}. Translating everything to our context, we obtain the following result.

\begin{tw}\label{THMabootUFRowi}
	There exists a unique, up to isomorphism, countable, finitely retractable graph $\ufr$ with the following properties.
	\begin{enumerate}[itemsep=0pt]
		\item[{\rm(1)}] Given nonempty finite graphs $A \subs B$ such that $A$ is a retract of $B$, given a left-invertible embedding $\map e A \ufr$, there exists a left-invertible embedding $\map f B \ufr$ such that $f \rest A = e$.
		\item[{\rm(2)}] Every countable, finitely retractable graph embeds into $\ufr$ as a retract.
		\item[{\rm(3)}] Given finite retracts $A_0, A_1$ of $\ufr$, every isomorphism between $A_0$ and $A_1$ extends to an automorphism of $\ufr$. In particular, for every $a,b \in \ufr$ there exists an automorphism $\map h{\ufr}{\ufr}$ such that $h(a)=b$.
	\end{enumerate}
\end{tw}

\begin{pf}[Sketch of proof.]
	(1) and (3) follow directly from the general \fra\ theory, see~\cite{K_flims}. The ``in particular'' part of (3) comes from the fact that a subgraph with one vertex is always a retract. Concerning (2), one needs to improve the setup by considering embedding-projections pairs, proving that the corresponding category has proper amalgamations, see~\cite[Section~6]{K_flims}. Without it, we would just obtain that every countable, finitely retractable graph is embeddable into $\ufr$, because in general the colimit of a sequence of left-invertible morphisms may not be left-invertible.
	Specifically, we are working with the category of all nonempty finite graphs with arrows being pairs of the form $\pair ep$, where $e,p$ are graph homomorphisms and $p \cmp e$ is an identity.
	Proper amalgamation means that given two embedding-projection pairs with the same domain, one can amalgamate them in such a way that the embeddings commute, the projections commute, and the ``mixed'' diagram consisting of parallel embeddings and parallel projections commutes as well.
	Our category has proper amalgamations, due to~\cite[Lemma~6.6]{K_flims}.
\end{pf}

What else can we say about this new graph? In Section~\ref{SectionRandomGraph} we shall learn that the random graph is finitely retractable, therefore it embeds into $\ufr$ as a retract. Hence, $\ufr$ is universal in the class of all countable graphs, namely, it contains an isomorphic copy of every countable graph.
On the other hand, retracts of $\ufr$ are precisely the nonempty countable finitely retractable graphs, due to Theorem~\ref{THMretsFinrectsha}.
Note that $\ufr$ is not isomorphic to the random graph, as it has infinitely many components, while $\rgraph$ is connected. Indeed, every countable graph with no non-trivial edges is a retract of $\ufr$. Furthermore, an infinite path is a retract, which shows that each component of $\ufr$ has infinite diameter, as retractions are non-expansive.

We now describe the envelope of $\ufr$.
Let $\catrigraphs$ denote the category of all nonempty finite graphs with right-invertible homomorphisms. This is a projective (inverse, or dual) version of the category above. Again, the amalgamation property holds trivially, the other axioms of the general \fra\ theory are satisfied, too. Thus we obtain the profinite \fra\ limit of $\catrigraphs$, which will be denoted by $\pufr$. More precisely:

\begin{tw}\label{THMaboutPufRha}
	There exists a unique, up to a topological isomorphism, profinite graph $\pufr$ with the following properties.
	\begin{enumerate}[itemsep=0pt]
		\item[{\rm(1)}] $\pufr$ is the limit of a sequence of right-invertible homomorphisms between finite nonempty graphs. 
		\item[{\rm(2)}] Given a right-invertible homomorphism $\map f BA$ between finite nonempty graphs, for every right-invertible continuous homomorphism $\map g{\pufr}B$ there exists a right-invertible continuous homomorphism $\map h{\pufr}A$ such that $f \cmp h = g$.
		$$\begin{tikzcd}
			A & & & \pufr \ar[lll, "g"'] \ar[dlll, dashed, "h"]\\
			B \ar[u, "f"] & & &
		\end{tikzcd}$$
		\item[{\rm(3)}] Let $K$ be a profinite graph that can be presented as the limit of an inverse sequence of right-invertible homomorphisms between finite graphs. Then there exist a topological graph embedding $\map e K \pufr$ and a continuous graph homomorphism $\map p \pufr K$ such that $p \cmp e = \id K$.
	\end{enumerate}
\end{tw}

The last property is typically called \emph{projective universality}, here restricted to a suitable class of profinite graphs. Graph $\pufr$ has also a \emph{projective homogeneity}, restricted to right-invertible continuous homomorphisms, namely, given an isomorphism $\map h{F_0}{F_1}$ between finite graphs, given right-invertible continuous homomorphisms $\map {f_i} \pufr {K_i}$, $i=0,1$, there is a topological graph automorphism $\map {\tilde h} \pufr \pufr$ such that the following diagram is commutative.
$$\begin{tikzcd}
	K_0 \ar[d, "h"', "\iso"] \ar[rr, "f_0"] & & \pufr \ar[d, "\tilde{h}", "\iso"'] \\
	K_1 \ar[rr, "f_1"'] & & \pufr
\end{tikzcd}$$
Theorem~\ref{THMaboutPufRha} can be justified in the same way as~\ref{THMabootUFRowi}, using the methods of~\cite[Section 6]{K_flims}.
Actually, as mentioned in the proof of Theorem~\ref{THMabootUFRowi}, these methods use the category $\catepgraphs$ of embedding-projection pairs between finite nonempty graphs. In particular, both $\ufr$ and $\pufr$ can be obtained at the same time, from a \fra\ sequence in $\catepgraphs$. Finally, we have:

\begin{tw}\label{THMPufrKoperta}
	$\pufr$ is an envelope of $\ufr$. More precisely, there exists a chain $\ciag F$ of nonempty finite induced subgraphs of $\pufr$, together with a sequence $\ciag r$ of continuous retractions of $\pufr$ with the following properties.
	\begin{enumerate}[itemsep=0pt]
		\item[{\rm(1)}] For each $\ntr$, $r_n$ is a graph homomorphism onto $F_n$.
		\item[{\rm(2)}] $r_n \cmp r_m = r_{\min(n,m)}$ for every $n,m \in \nat$.
		\item[{\rm(3)}] $\bigcup_{\ntr}F_n$ is isomorphic to $\ufr$.
	\end{enumerate}
\end{tw}

\begin{pf}[Sketch of proof]
	Concerning (2), see Proposition~\ref{PROPSztyryDwa} and its proof. The remaining properties follow from the fact that, given a \fra\ sequence in $\catepgraphs$, forgetting one of the components (either the embeddings or the epimorphisms), we obtain a \fra\ sequence in $\catrigraphs$ or $\catligraphs$. Hence the colimit of its ``embedding'' component is $\ufr$ and the limit of its ``projection'' component is $\pufr$.
\end{pf}

\begin{uwgi}
	One can also consider a connected variant of the story described above, namely, restricting the objects to connected graphs. The standard (free) amalgamation works and every retract of a connected graph is connected. Denote by $\ufrctd$ and $\pufrctd$ the corresponding \fra\ limits, the discrete one and the profinite one. It is easy to see that $\ufrctd$ is isomorphic to any of the infinitely many components of $\ufr$ (just checking the connected variant of Theorem~\ref{THMabootUFRowi}(1)).
	An analog of Theorem~\ref{THMPufrKoperta} holds, however $\pufrctd$ is no longer connected, as the result below shows. 
\end{uwgi}

It turns out that $\pufr$ is not homogeneous, in fact its automorphism group (consisting of all topological graph automorphisms) does not act transitively on $\pufr$. The same is true for its ``connected" counterpart $\pufrctd$.
This is an immediate consequence of the next result.

\begin{tw}\label{THMieuhrouwrgouwg}
	Each of the graphs $\pufr$ and $\pufrctd$ has a dense $G_\delta$ set of isolated vertices.
\end{tw}

\begin{pf}
	We check condition (IV) of Lemma~\ref{THMKriterionIsoVerticess}.
	Namely, fix a nonempty open set $U \subs \pufr$ (respectively, $\pufrctd$) and choose a right-invertible continuous epimorphism $\map g \pufr A$ (resp. $\map g \pufrctd A$) onto a finite (connected) graph, such that $g^{-1}(v) \subs U$ for some $v \in A$.
	
	Let $B = A \cup \sn w$, where $w \notin A$ is adjacent to $v$ and to no other vertex of $A$. If $A$ was connected, so is $B$.
	Let $\map r BA$ be such that $r(w)=v$ and $r\rest A = \id A$. Clearly, $r$ is a retraction.
	By Theorem~\ref{THMaboutPufRha}(1), there is a right-invertible continuous homomorphism $\map h \pufr B$ (resp. $\map h \pufrctd B$) such that $g =  r \cmp h$.
	Let $W = h^{-1}(w)$. Then
	$$W \subs h^{-1}\dn vw = h^{-1} r^{-1}(v) = {(r \cmp h)}^{-1}(v) = g^{-1}(v) \subs U.$$
	If $x \in \pufr \setminus U$ (resp. $\pufrctd \setminus U$) then $g(x) \ne v$, therefore $h(x) \notin \dn vw$ and consequently $h(x)$ is not adjacent to $w$.
	Hence, $x$ is adjacent to no point in $W = h^{-1}(w)$. This shows that there are no adjacencies between $\pufr \setminus U$ (resp. $\pufrctd \setminus U$) and $W$, proving (IV) of Lemma~\ref{THMKriterionIsoVerticess}.	
\end{pf}

We now see that the limit of an inverse sequence of retractions between finite connected graphs may be a disconnected graph, with many isolated vertices. On the other hand, it always contains a topologically dense countable connected graph, namely, the union of the corresponding chain (see Proposition~\ref{PROPdensitySixkm}).

\section{Graph evolutions and sociable graphs}

We now discuss a natural way of building a graph step-by-step, starting from a single vertex. Just adding a new vertex is fairly general and can lead to any graph. Duplicating a fixed vertex is another way, however, it might be too restrictive, as long as we agree to keep the existing connections. We propose to look at processes where at each step a new vertex is first duplicated from an existing one and then some of the connections break, as visualized in the following figure.

$$
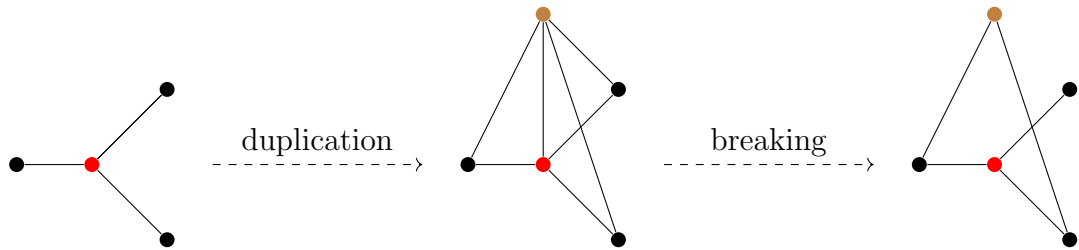
\begin{tikzpicture}[every node/.style={circle,inner sep=2pt}]
				\node[fill=black] (a0) at (0,1) {};
				\node[fill=red] (v) at (1,1) {};
				\node[fill=black] (a1) at (2,2) {};
				\node[fill=black] (a2) at (2,0) {};
				\draw (a0) -- (v);
				\draw (a1) -- (v);
				\draw (a2) -- (v);
				\draw (a1) -- (v);
				
				\node[fill=white] (fake1) at (2.5,1) {};
				\node[fill=white] (fake2) at (5.5,1) {};
				\draw[->,dashed]      (fake1) -- (fake2);
				\node at (4,1.3) {duplication}; 
				
				\node[fill=black] (b0) at (6,1) {};
				\node[fill=red] (v1) at (7,1) {};
				\node[fill=black] (b1) at (8,2) {};
				\node[fill=black] (b2) at (8,0) {};
				\node[fill=brown] (v2) at (7,3) {};
				\draw (b0) -- (v1);
				\draw (b1) -- (v1);
				\draw (b2) -- (v1);
				\draw (b0) -- (v2);
				\draw (b1) -- (v2);
				\draw (b2) -- (v2);
				\draw (v2) -- (v1);
				
				\node[fill=white] (fake1) at (8.5,1) {};
				\node[fill=white] (fake2) at (11.5,1) {};
				\draw[->,dashed]      (fake1) -- (fake2);
				\node at (10,1.3) {breaking}; 
				
				\node[fill=black] (c0) at (12,1) {};
				\node[fill=red] (v3) at (13,1) {};
				\node[fill=black] (c1) at (14,2) {};
				\node[fill=black] (c2) at (14,0) {};
				\node[fill=brown] (v4) at (13,3) {};
				\draw (c0) -- (v3);
				\draw (c1) -- (v3);
				\draw (c2) -- (v3);
				\draw (c0) -- (v4);
				\draw (c2) -- (v4);
\end{tikzpicture}$$
Note that mapping the brown vertex to the red one is a retraction. The edges between the red vertex and the black ones have to be preserved, while some of the edges resulting from the duplication can disappear.

Another possible motivation for considering processes as above are social (perhaps authoritarian) networks, where a new member $w$ can be added only if she/he is invited by a specified member $v$ of the existing network $G$ and this can happen only if the new member $w$ does not have connections (e.g. friends, or some other dependencies, encoded by the edge relation) outside of the neighborhood of $v$. In other words, if $x$ is connected to $w$ and belongs to the existing network, then $x$ is connected to $v$. Mapping $w$ to $v$ provides a natural retraction from $G \cup \sn w$ to $G$. In this example, one should perhaps require that $w$ be connected to $v$, that is, $v$ is not willing to invite anyone who is not already connected to her/him.

Iterating the steps as above, starting from the simplest nonempty graph, we obtain the concept of a \emph{graph evolution}. Evolutions as above might possibly serve as models of some physical, social, or biological processes.

\separator

It should be clear that both examples give particular cases of embedding-projection pairs $\pair ep$, where $\map e GH$ is the inclusion $G \subs H$, $H = G \cup \sn w$, and $p$ is a retraction, mapping $w$ to a specific vertex $v \in G$. In the first example, $v$ was first duplicated to $w$ and then certain connections were removed, while in the second example $v$ ``invites'' a new vertex $w$, whose neighborhood is contained in that of $v$.
Below is the precise definition.

\begin{df}\label{DEFonebyOneTwoTree}
	 (1) Given a finite nonempty graph $G$, a \emph{transition} from $G$ to a graph $H$ is an embedding-projection pair $\map{\pair ep}{G}{H}$ (so $\map eGH$ is a graph embedding, $\map pHG$ is a graph homomorphism and $p \cmp e = \id G$) such that $H \setminus \img eG$ is either empty (then it is called \emph{trivial}) or consists of a single vertex $w$ (then $\pair ep$ is \emph{nontrivial}). 
	 
	 If either $\pair ep$ is trivial or the unique vertex $w \in H \setminus \img eG$ is connected to $e(p(w))$, we will say that $\pair ep$ is \emph{sociable}.
	 
	 (2) We shall denote by $\origin$ the graph with a single vertex. This is the \emph{origin} of all the graph evolutions.
	 
	 (3) A \emph{graph evolution} will be any composition of transitions, namely, an embedding-projection pair resulting from a sequence
	 $$\begin{tikzcd}
	 	G = G_0\arrow[r, hook,shift right=1.5, "e_{0}"']& G_1\arrow[l, shift right=1.5, "p_{0}"']\arrow[r,hook,shift right=1.5, "e_{1}"']&  G_2\arrow[l, shift right=1.5, "p_{1}"']\dots G_{n-1}\arrow[r,hook,shift right=1.5,"e_{n-1}"']&\arrow[l, shift right=1.5,"p_{n-1}"']G_n = H,
	 \end{tikzcd}$$
	 where each $\pair{e_i}{p_i}$ is a transition. Formally, this composition is $$\pair{e_{n-1} \cmp \dots \cmp e_0}{p_{0} \cmp \dots \cmp p_{n-1}}.$$
	 An \emph{infinite graph evolution} is defined similarly, the main difference is that $H$ becomes infinite (if each $e_i$ is the inclusion $G_i \subs G_{i+1}$, then $H = \bigcup_{\ntr}G_n$) and on the ``projection" one can consider the inverse limit---the envelope of $H$.
	 A (finite) graph evolution $f = \pair ep$ from $G$ to $H$ will be denoted by $\ewa f G H$.
	 
	 (4) An evolution will be called \emph{sociable} if it is formed by sociable transitions. A graph $W$ (finite or countable infinite) will be called \emph{sociable} if there exists a sociable evolution from the origin $\origin$ to $W$.
\end{df}

\begin{example}
	\label{ExampleEvolution} Let us illustrate the concept of a graph evolution. 
	
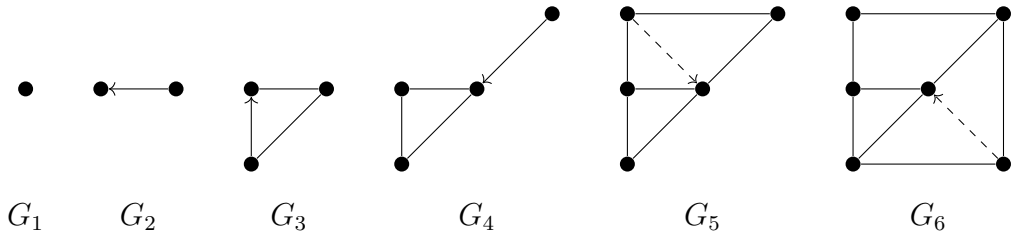
\begin{equation*}
		\begin{tikzpicture}[every node/.style={circle,inner sep=2pt}]
			
			\node[fill=black] (a0) at (-1,1) {};
			\node at (-1,-0.7) {$G_1$}; 
			
			\node[fill=black] (a1) at (0,1) {};
			\node[fill=black] (b1) at (1,1) {};  
			\node at (0.5,-0.7) {$G_2$};

			\draw[->] (b1) -- (a1);
			
			\node[fill=black] (a2) at (2,1) {};
			\node[fill=black] (b2) at (3,1) {};  
			\node[fill=black] (c2) at (2,0) {};    
			\node at (2.5,-0.7) {$G_3$}; 
			
			\draw (a2) -- (b2)
			(b2) -- (c2);
			\draw[->]      (c2) -- (a2);
			
			\node[fill=black] (a3) at (4,1) {};
			\node[fill=black] (b3) at (5,1) {};  
			\node[fill=black] (c3) at (4,0) {};
			\node[fill=black] (d3) at (6,2) {};      
			\node at (5,-0.7) {$G_4$}; 
			
			\draw (a3) -- (b3)
			(b3) -- (c3)
			(c3) -- (a3);
			\draw[->]      (d3) -- (b3);	
			
			\node[fill=black] (a4) at (7,1) {};
			\node[fill=black] (b4) at (8,1) {};  
			\node[fill=black] (c4) at (7,0) {};
			\node[fill=black] (d4) at (9,2) {};
			\node[fill=black] (e4) at (7,2) {};        
			\node at (8,-0.7) {$G_5$};
			
			\draw (a4) -- (b4)
			(b4) -- (c4)
			(c4) -- (a4)
			(d4) -- (b4)
			(e4) -- (a4)
			(e4) -- (d4);
			\draw[->,dashed]      (e4) -- (b4);
			
			\node[fill=black] (a5) at (10,1) {};
			\node[fill=black] (b5) at (11,1) {};  
			\node[fill=black] (c5) at (10,0) {};
			\node[fill=black] (d5) at (12,2) {};
			\node[fill=black] (e5) at (10,2) {};        
			\node[fill=black] (f5) at (12,0) {};        
			\node at (11,-0.7) {$G_6$};
			
			\draw (a5) -- (b5)
			(b5) -- (c5)
			(c5) -- (a5)
			(d5) -- (b5)
			(e5) -- (a5)
			(e5) -- (d5)
			(f5) -- (c5)
			(f5) -- (d5);
			\draw[->,dashed]      (f5) -- (b5);	
			
		\end{tikzpicture}
	\end{equation*}
	Graph $G_i$ embeds in $G_{i+1}$ in the obvious way. The arrows show how $G_{i+1}$ projects onto $G_i$. Note that $G_6$ contains a graph isomorphic to $C_5$.
	By Proposition~\ref{PROPCykleCycles} below, we know that $C_5$ is not PPR. This shows that a PPR finite graph can contain $C_5$ as an induced subgraph.
		
	Note that $G_6$ can also be decomposed back as follows
	
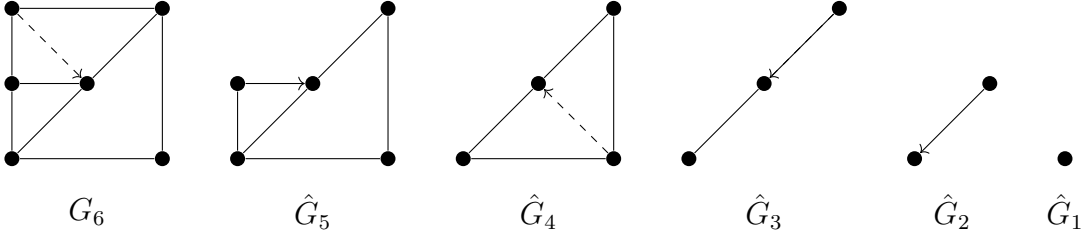
\begin{equation*}
		\begin{tikzpicture}[every node/.style={circle,inner sep=2pt}]
			
			\node[fill=black] (a6) at (0,1) {};
			\node[fill=black] (b6) at (1,1) {};  
			\node[fill=black] (c6) at (0,0) {};
			\node[fill=black] (d6) at (2,2) {};
			\node[fill=black] (e6) at (0,2) {};        
			\node[fill=black] (f6) at (2,0) {};        
			\node at (1,-0.7) {$G_6$};
			
			\draw (a6) -- (b6)
			(b6) -- (c6)
			(c6) -- (a6)
			(d6) -- (b6)
			(e6) -- (a6)
			(e6) -- (d6)
			(f6) -- (c6)
			(f6) -- (d6);
			\draw[->,dashed]      (e6) -- (b6);

			\node[fill=black] (a5) at (3,1) {};
			\node[fill=black] (b5) at (4,1) {};  
			\node[fill=black] (c5) at (3,0) {};
			\node[fill=black] (d5) at (5,2) {};
			\node[fill=black] (f5) at (5,0) {};        
			\node at (4,-0.7) {$\hat{G}_5$};
			
			\draw (b5) -- (c5)
			(c5) -- (a5)
			(d5) -- (b5)
			(f5) -- (c5)
			(f5) -- (d5);
			\draw[->]      (a5) -- (b5);

			\node[fill=black] (b4) at (7,1) {};  
			\node[fill=black] (c4) at (6,0) {};
			\node[fill=black] (d4) at (8,2) {};
			\node[fill=black] (f4) at (8,0) {};        
			\node at (7,-0.7) {$\hat{G}_4$};
			
			\draw (b4) -- (c4)
			(d4) -- (b4)
			(f4) -- (c4)
			(f4) -- (d4);
			\draw[->,dashed]      (f4) -- (b4);

			\node[fill=black] (b3) at (10,1) {};  
			\node[fill=black] (c3) at (9,0) {};
			\node[fill=black] (d3) at (11,2) {};
			\node at (10,-0.7) {$\hat{G}_3$};
			
			\draw (b3) -- (c3)
			(d3) -- (b3);
			\draw[->,dashed]      (d3) -- (b3);

			\node[fill=black] (b2) at (13,1) {};  
			\node[fill=black] (c2) at (12,0) {};
			\node at (12.5,-0.7) {$\hat{G}_2$};
			
			\draw[->]      (b2) -- (c2);

			\node[fill=black] (c2) at (14,0) {};
			\node at (14,-0.7) {$\hat{G}_1$};
			
		\end{tikzpicture}
	\end{equation*}
	A natural question arises here: if one chooses a way of decomposing the graph, different from the original one, is it possible to get stuck? More formally, suppose that $G$ is finite and PPR, $\map p G {G\setminus\{v\}}$ is a retraction; is $G\setminus\{v\}$ PPR as well? An affirmative answer to this question follows from Theorem~\ref{THMRetraktyFinitoPoPaPr}, giving us an algorithm for checking whether a given finite graph is point-by-point retractable. We shall see in a moment that the same applies to sociable graphs.	
\end{example}

Graph evolutions might possibly model some physical, social, or biological processes.

Definition~\ref{DEFonebyOneTwoTree} actually gives us three classes of countable graphs, those obtained by evolutions from $\origin$, the connected ones, and the sociable ones. Note that disconnected graphs are obtained by evolutions in which at least one of the transitions adds an isolated vertex---this is because the union of any chain of connected graphs is connected.

It is rather obvious that every sociable graph is connected, as at each step we are adding a vertex adjacent to one of the existing ones. The following fact follows from the definition of transitions and Proposition~\ref{PROPSztyryDwa}.

\begin{prop}
	A nonempty countable graph $G$ is point-by-point retractable if and only if there exists a graph evolution from $\origin$ to $G$.
\end{prop}

The class of sociable graphs is strictly smaller than the class of connected PPR graphs, as the following easy fact shows.

\begin{prop}\label{PROPCykleCycles}
	Let $k>2$ be a natural number. The $k$-element cycle $C_k$ is PPR if and only if $k<5$. It is sociable if and only if $k=3$.
\end{prop}

\begin{pf}
	The cycle $C_k$ with $k \goe 5$ does not admit any retraction onto $C_k \setminus \sn v$, where $v$ is any fixed vertex of $C_k$. The cycle $C_4$ is evidently PPR, however its unique retraction onto $C_4\setminus \sn v$, where $v \in C_4$, is not sociable, as it has to map $v$ to a vertex along the diagonal, not adjacent to $v$. Obviously, every nonempty graph with at most $3$ vertices, including $C_3$, is sociable.
\end{pf}

Thus, $C_4$ is the smallest connected PPR graph that is not sociable.
Adapting the proof of Theorem~\ref{THMRetraktyFinitoPoPaPr}, we obtain

\begin{tw}
	Every finite retract of a sociable graph is sociable.
\end{tw}

\begin{pf}
	Fix a finite graph $G$, a retract of a sociable graph. We may assume that $G \subs H$, where $H$ is sociable of minimal cardinality such that there is a retraction $\map r H G$. Let $H = H' \cup \sn v$, where $H'$ is sociable and there is a retraction $\map s H H'$. By minimality, $v \in G$ and $r(s(v)) \in G'$, since otherwise, by the proof of Theorem~\ref{THMRetraktyFinitoPoPaPr}, we would get a contradiction with the minimality of $H$. Note that $r(s(v))$ is adjacent to $v$, because $s(v)$ is adjacent to $v$ ($s$ is part of a sociable transition) and $r$ preserves the edges.
	Finally, $G'$ is a retract of a smaller sociable graph $H'$, therefore by the inductive hypothesis it is sociable.
\end{pf}

\subsection{The most complicated sociable graph}

Let us look at the category of nonempty (connected) finite PPR graphs, namely, those for which there is an evolution from $\origin$ (in the connected case, consisting of connected subgraphs). Clearly, the free amalgamation of two transitions with the same domain consists of transitions (see~\cite[Lemma 6.6]{K_flims}), therefore by easy induction we infer that finite evolutions have the proper amalgamation property. Hence, the general \fra\ theory of embedding-projection pairs gives us a unique countable PPR graph $\uppr$ (and its connected variant $\upprctd$), characterized by the following extension property.
\begin{enumerate}
	\item[(E)] Given a finite (connected) PPR graph $A$ and an evolution $\ewa f A \uppr$ (resp. $\ewa f A \upprctd$), given a transition $A \subs B$ (where $B$ is connected), there exists an evolution from $B$ to $\uppr$ (resp. $\upprctd$) extending $f$.
\end{enumerate}
Furthermore, $\uppr$ and $\upprctd$ have similar properties as $\ufr$ and $\ufrctd$, namely, universality and a rather special variant of homogeneity. Universality says that every (connected) PPR graph is a retract of $\uppr$ (resp. $\upprctd$).
Again, it is easy to see that $\uppr$ has infinitely many components, each of them being isomorphic to $\upprctd$.

The graphs $\uppr$ and $\upprctd$ are obtained through the corresponding categories of embedding-projection pairs, therefore each of them has a natural envelope, denoted by $\puppr$ and $\pupprctd$, respectively.
Note that given a (connected) finite graph $A$ with a distinguished vertex $v$, there is a sociable transition from $A$ to $A \cup \sn w$ such that $w$ is adjacent precisely to $v$ (and to itself). Hence, the same argument as in the proof of Theorem~\ref{THMieuhrouwrgouwg} gives the following

\begin{tw}
	Each of the profinite graphs $\puppr$ and $\pupprctd$ has a dense $G_\delta$ set of isolated vertices.
\end{tw}

Now, let us look at the category of sociable graphs. Here, the evolutions consist of sociable transitions and again the free amalgamation works, namely it consists of sociable transitions.
Thus, the general theory provides a countable graph $\soc$, characterized uniquely by a sociable variant of the extension property (E).
The same argument as above shows that its natural envelope $\psoc$ has a dense $G_\delta$ set of isolated vertices.
Actually, even restricting to the simplest possible sociable transitions we obtain a profinite graph with a dense set of isolated vertices. This is explained below.

\begin{example}
	A sociable transition from a graph $A$ to a graph $B$ will be called \emph{simple} if the new vertex $w \in B \setminus A$ (assuming the embedding is inclusion) is adjacent to a single vertex in $A$.
	Starting from the one-vertex graph $\origin$ and using compositions of simple transitions (together with isomorphisms), we obtain the class of all nonempty finite trees (recall that a \emph{tree} is a connected cycle-free graph). The category of finite trees with evolutions consisting of simple transitions clearly has proper amalgamations and its \fra\ limit leads to the unique tree $\mathbf T$ whose each vertex has infinite degree. Its natural envelope $\ovr{\mathbf T}$ has a dense $G_\delta$ set of isolated vertices, because the sociable transition described above is actually simple.
\end{example}

\section{The random graph is sociable}\label{SectionRandomGraph}

In this section we give a criterion for being sociable, showing in particular that all retracts of the random graph are in this class.

We say that a vertex $u$ in a graph $G$ is a \emph{sentinel}, if it is adjacent to each vertex of $G$.

\begin{lm}\label{LMeourghoirg}
	Assume a graph $G = \bigcup_{\ntr} G_n$ is such that $\ciag G$ is a chain of finite induced subgraphs. If each $G_n$ has a sentinel then $G$ is sociable. 
\end{lm}

\begin{pf}
	We may assume that $G_0$ has a single vertex, adding if necessarily a one-vertex subgraph of $G_0$ to the chain.
	Fix $n$ and assume $G_{n+1} \setminus G_n = \{v_0, \dots, v_{k-1}\}$, where $v_0$ is a sentinel in $G_{n+1}$. It is enough to show that, setting $G_n^i = G_n \cup \{v_0, \dots, v_{i-1}\}$, we can find a suitable retraction from $G_n^{i+1}$ onto $G_n^i$ for each $i \loe k$.
	
	There is a retraction from $G_n^1$ onto $G_n = G_n^0$, mapping $v_0$ to a sentinel in $G_n$. Clearly, $v_0$ is a sentinel in each $G_n^i$, therefore mapping $v_i$ to $v_{0}$ we obtain a retraction from $G_n^{i+1}$ onto $G_n^i$. Furthermore, $v_i$ is connected to $v_0$ for each $i$, therefore the retraction forms a sociable transition. This completes the proof.	
\end{pf}

\begin{wn}
	The random graph is sociable.
\end{wn}

\begin{pf}
	Easy induction gives a chain $\sett{G_n}{\ntr}$ of finite subgraphs of $\rgraph$ such that $\rgraph = \bigcup_\ntr G_n$ and each $G_n$ has a sentinel. Indeed, we first enumerate $\rgraph = \sett{v_n}{\ntr}$, define $G_0 = \sn{v_0}$ and having defined $G_{n-1}$, we define $G_n = G_{n-1} \cup \dn{v_n}{w}$, where $w$ is adjacent to all vertices in $G_{n-1} \cup \sn{v_n}$.	
\end{pf}

The proof above actually shows a more general statement.

\begin{wn}
	Assume $G$ is a countable graph such that for every finite $A \subs G$ there exists a vertex adjacent to all elements of $A$. Then $G$ is sociable.
	
	In particular, every retract of the random graph is sociable.
\end{wn}

We now give an explicit construction of a sequence of retractions on the random graph, obtaining a universal profinite graph.

\begin{tw}
There is a closed graph $G$ on $2^\omega$ without isolated vertices (in the sense of graph theory) that contains a dense induced copy of the random graph $\rgraph$. In particular, every countable graph embeds into a closed graph on $2^\omega$. 
\end{tw}

\begin{proof}
Let $\{x_n:n\in\omega\}$ be an enumeration of all vertices of the random graph $\rgraph$. We will define graphs $G_0\subseteq G_1\subseteq G_2\subseteq\dots$ and projections $\map {p_i}{G_{i+1}}{G_i}$ such that 
\begin{itemize}
    \item [(1)] $G_i$ is induced subgraph of $\rgraph$ containing $x_i$;
    \item [(2)] for any $x\in G_i$ there is $y\in G_{i+1}\setminus G_i$ with $p_i(y)=x$;
    \item[(3)] $G_i$ has a sentinel $u_i$.
\end{itemize}
We start the construction from $G_0=\{x_0\}$. Assume that we have already constructed $G_0,G_1,\dots,G_i$. In the random graph $\rgraph$ we find an isomorphic disjoint copy of $G_i$, say $X$, such that any element of $G_i$ is not adjacent to any element of $X$. Let $\map fX{G_i}$ be a graph isomorphism. Let $v$ be a vertex in $\rgraph$ adjacent with every vertex in the finite set $G_i\cup X\cup\{x_{i+1}\}$. Put $G_{i+1}:=G_i\cup X\cup\{x_{i+1},v\}$. It is an induced subgraph of $\rgraph$. Define $\map {p_i}{G_{i+1}}{G_i}$ as follows
\[
p_{i}(y):=\left\{\begin{array}{ccc}
y&\mbox{if}&y\in G_{i};\\
f(y)&\mbox{if}&y\in X;\\
u_i&\mbox{if}&\text{otherwise}.
\end{array}\right.
\]
Note that if $x_{i+1}\in G_i$, then $p_i(x_{i+1})=x_{i+1}$; if $x_{i+1}\in X$, then $p_i(x_{i+1})=f(x_{i+1})$; if $x_{i+1}\notin G_i\cup X$, then $p_i(x_{i+1})=u$. It is easy to see that $p_i$ is a retraction. Moreover, $G_i$ fulfills (1)--(3). The direct limit of $G_i$'s is isomorphic to $\bigcup_{i\in\omega}G_i=\rgraph$ and it is dense in the inverse limit $G$ of $\langle (G_i)_{i\in\omega}, (p_i)_{i\in\omega} \rangle$. Since for every $i\in\omega$ and every $x\in G_i$ there is $y\in G_{i+1}\setminus G_i$ such that $p_i(y)=x=p_i(x)$, then $G$ is homeomophic to $2^\omega$.  

By Lemma~\ref{Weakly3SaturatedProfiniteHasNoIsolatedPoints} any closed graph on $2^\omega$ that contains a dense induced copy of $\rgraph$ has no isolated vertices (in the graph-theoretic sense).
\end{proof}

Since the random graph is isomorphic to its complement, we also get the following:

\begin{wn}
	There is an open graph $G$ on $2^\omega$ that contains a dense induced copy of the 
	random graph.
	In particular, every countable graph embeds into an open graph on $2^\omega$.
\end{wn}

\section{Henson's graphs are not finitely retractable}\label{SectionHensonGraphs}

Recall that \emph{Henson's universal $K_n$-free graph} is the \fra\ limit of finite $K_n$-free graphs, namely, the unique countable homogeneous $K_n$-free graph containing isomorphic copies of all countable $K_n$-free graphs. We denote this graph by $\Hen n$.
Recall that $K_n$ denotes, as usual, the complete graph with $n$ vertices.

\begin{tw}
Given $\ell>2$, the graph $\Hen \ell$ is not finitely retractable.
\end{tw}

\begin{pf}
We first present the arguments for $\ell=3$.
Namely, let $G_0$ be the graph with vertices $a,b,c,d$ and nontrivial edges $a - c$, $b - d$ (recall that we consider graphs with loops).
As $G_0$ is $K_3$-free, we may assume that it is contained in $\Hen3$. 
We claim that $G_0$ is not contained in any finite complemented subgraph of $\Hen3$.
Namely, suppose $H \subs \Hen3$ is a complemented subgraph such that $G_0 \subs H$.

Define $A_0 = \{ a, b \}$, $B_0 = \{ c, d \}$.
We claim that there exist vertices $p_0,q_0$ in $H$ such that $p_0$ is adjacent precisely to $A_0$ and $q_0$ is adjacent precisely to $B_0$.
Indeed, such vertices must exist in $\Hen3$ and therefore their images under a projection onto $H$ satisfy the same requirements. It is important that $p_0 \ne q_0$ and $p_0, q_0 \notin G_0$.
We now set $G_1 = G_0 \cup \dn {p_0}{q_0}$ and
$$A_1 = A_0 \cup \sn {q_0}, \qquad B_1 = B_0 \cup \sn {p_0}.$$
Note that the sets $A_1$, $B_1$ are independent and the same argument as above yields new vertices $p_1, q_1 \in H$ such that $p_1$ is adjacent precisely to $A_1$ and $q_1$ is adjacent precisely to $B_1$.
Again, $p_1 \ne q_1$.
We continue like this, obtaining a strictly increasing chain
$$G_0 \subs G_1 \subs G_2 \subs \dots \subs H$$
showing that $H$ is infinite.

Now assume $\ell>2$ is arbitrary.
Given a graph $G$, we say that its subset $A$ is \emph{accessible} in $G$ if it is $K_{\ell-1}$-free and no vertex in $G$ is adjacent
to all elements of $A$.
By a \emph{maximal accessible set} we mean a set that is accessible and not contained in a bigger accessible set. Note that some graphs (e.g. $K_{\ell-2}$) do not contain nonempty accessible sets.
On the other hand, every at least 2-element $K_2$-free set is accessible in every graph.

We now fix a finite graph $G_0 \subs \Hen \ell$ with two different maximal accessible sets $A_0, B_0$, each of them containing a copy of $K_{\ell-2}$.
Later we show the existence of such a graph.
We claim that no finite graph containing $G_0$ is complemented in $\Hen \ell$.
For this aim, we suppose that $G_0 \subs H$, where $H$ is complemented in $\Hen \ell$ and we construct inductively a strictly increasing chain
$$G_0 \subs G_1 \subs G_2 \subs \dots \subs H$$
so that each $G_n$ contains two different maximal accessible sets $A_n,B_n$, each of them containing a copy of $K_{\ell-2}$.
This will show that $H$ is infinite.

Assume $G_n$ together with $A_n, B_n$ as above is given.
Note that there exists $p \in H$ that is adjacent to all elements of $A_n$ (and, by maximality, to no other elements of $G_n$).
This follows from the extension property of $\Hen \ell$ and from the fact that $H$ is complemented in $\Hen \ell$.
Similarly, there exists $q\in H$, adjacent to all elements of $B_n$ (and to no other elements of $G_n$).
Necessarily, $p \ne q$ and $p,q \notin G_n$, just by the definition of an accessible set.
Define $G_{n+1} = G_n \cup \dn p q$ and $A_{n+1} = A_n \cup \sn q$, $B_{n+1} = B_n \cup \sn p$.
Note that $A_{n+1}$, $B_{n+1}$ are accessible in $G_{n+1}$, because $p$ and $q$ are isolated vertices of $B_{n+1}$ and $A_{n+1}$, respectively.
Note also that $A_{n+1}\cup \sn p$ is not contained in any accessible set in $G_{n+1}$, because it contains a copy of $K_{\ell-1}$ (as $A_n \subs A_{n+1}$ already contains it).
Similarly, $B_{n+1}\cup \sn q$ is not contained in any accessible subset of $G_{n+1}$. In other words, $A_{n+1}$, $B_{n+1}$ are maximal accessible sets.
This shows that the construction can be carried out.
Finally, it remains to find $G_0$.

It is easy find $G_0$ when $\ell = 3$ (see above), so assume $\ell > 3$.
Let $G_0$ be the graph obtained from $K_{\ell-1}$ by removing exactly one edge.
Denote by $a, b$ the unique vertices that are not adjacent in $G_0$. 
Choose distinct vertices $c, d \in G_0 \setminus \dn a b$ and define $A_0 = G_0 \setminus \sn c$, $B_0 = G_0 \setminus \sn d$.
Observe that $A_0$, $B_0$ are admissible: $K_{\ell-1}$ is not contained in $G_0$ and no vertex is adjacent to all elements of $A_0$ or $B_0$.
They are also maximal with these properties, because the only missing vertex is adjacent to all elements of $G_0$.
Finally, $A_0 \cap B_0$ has a copy of $K_{\ell-2}$.
Thus, $G_0$ has all the required properties, which completes the proof.
\end{pf}

\section{A concrete \fra\ sequence for graph evolutions}

Below we describe a concrete construction of the \fra\ limit $\upprctd$ of connected finite PPR graphs.

Let $G$ be a graph and let $x\in G$ be its fixed vertex. By $N_{G}(x)$ we denote the neighborhood of $x$, that is the set of all vertices $y\in G\setminus\{x\}$ such that $\pair xy$ is an edge in $G$. We will write $N(x)$ instead of $N_{G}(x)$ if the graph structure is clear from the context. 

\begin{lm}\label{1PointExtension} Assume that $G$ is a subgraph of $H$, $H\setminus G=\{x\}$ and $\map pHG$ is a retraction. Then $N_{H}(x)\subseteq N_{G}(p(x))\cup\{p(x)\}$, that is, if $x$ and $y$ are adjacent for some $y\neq x$, then so are $p(x)$ and $y$.  
\end{lm}

\begin{proof}
	Let $z\in N_{H}(x)$. Since $z\neq x$, then $z\in G$. Since $p$ preserves edges, then $\{p(x),p(z)\}$ is an edge in $G$. But $p(z)=z$, so $z\in N_{G}(p(x))\cup\{p(x)\}$. 
\end{proof}

From Lemma \ref{1PointExtension} we immediately obtain that if $G$ is a retract of its one-vertex extension $H$, then $\diam G\leq \diam H\leq \diam G+1$. 

\begin{lm}\label{DecompositionLemma}
	Assume that $G$ is a subgraph of a finite graph $H$. Assume that $\map pHG$ is a retraction such that $N_H(x)\subseteq N_{G}(p(x))\cup\{p(x)\}$ for every $x\in H$. Then there is a chain of one-vertex extensions
	\[
	\begin{tikzcd}
		G=F_0&\arrow[l, "p_{0}"'] F_1&\arrow[l, "p_{1}"']F_2&\arrow[l,]\cdots&\arrow[l,"p_{k-1}"'] F_k=H
	\end{tikzcd}
	\]
	such that each $p_i$ is a retraction. 
\end{lm}

\begin{proof} 
	Let $H\setminus G=\{y_1,\dots,y_k\}$. Let $F_0=G$ and $F_i=F_{i-1}\cup\{y_i\}$ for $i=1,\dots,k$ be induced subgraphs of $H$. Let $\map{p_{i-1}}{F_i}{F_{i-1}}$ be the identity on $F_{i-1}$ and $p_{i-1}(y_i):=p(y_i)$. 
	We need only to show that if
	$x\in F_{i-1}$ is such that $x$ and $y_i$ are adjacent in $F_i$, then $x=p(x)$ and $p(y_i)$ are adjacent in $F_{i-1}$. Since $N_H(y_i)\subseteq N_G(p(y_i))\cup\{p(y_i)\}$, then $x\in N_G(p(y_i))$ or $x=p(y_i)$. Thus $x$ and $p(y_i)$ are adjacent in $G$. \end{proof}

Let $G_1$ be a graph consisting of two vertices and the edge between them (and the loops). Inductively we define $G_1,G_2,\dots$ as follows. Suppose that we have already defined $G_k$. For any vertex $x\in G_k$ let $N(x)=\{y_1,\dots,y_l\}$. For any sequence $t=(t_0,t_1,\dots,t_l)\in\{0,1\}^{l+1}\setminus\{(0,0,\dots,0)\}$ by $x_t$ we denote a new vertex and we say that $x_t$ and $y_i$ are adjacent $\iff$ $t_i=1$, where $y_0=x$. There are no other new edges than those prescribed. After repeating the same construction for every vertex of $G_k$, we obtain $G_{k+1}$. We define $\map {p_{k}}{G_{k+1}}{G_k}$ by putting $p_{k}(x_t)=x$ and $p_{k}(x)=x$ for every vertex $x\in G_k$. Clearly $p_{k}$ is a quotient map and retraction.  

Below we present graphs $G_1$: 
\begin{center}
	\begin{tikzpicture}[every node/.style={circle,inner sep=2pt,fill=black}]
		\node[color=white] (A) at (0,1) {};
		\node[color=red] (B) at (1,1) {};
		\node (C) at (4,1) {};

		\draw 
		(B) -- (C);
	\end{tikzpicture}
\end{center}
and $G_2$: 
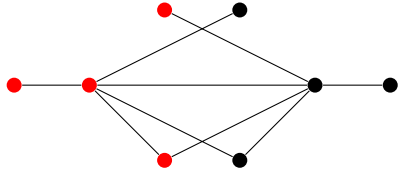
\begin{center}
	\begin{tikzpicture}[every node/.style={circle,inner sep=2pt,fill=black}]
		\node[color=red] (A) at (0,1) {};
		\node[color=red] (B) at (1,1) {};
		\node (C) at (4,1) {};
		\node[color=red] (D) at (2,0) {};
		\node (E) at (3,0) {};
		\node[color=red] (F) at (2,2) {};
		\node (G) at (3,2) {};
		\node (H) at (5,1) {};

		\draw (A) -- (B)
		(B) -- (C)
		(B) -- (D)
		(B) -- (E)
		(B) -- (G)
		(C) -- (D)
		(C) -- (E)
		(C) -- (F)
		(C) -- (H);
	\end{tikzpicture}
\end{center}

Black vertices in $G_2$ are mapped via $p_{1}$ onto the black vertex in $G_1$ and the same for red ones. Graph $G_3$ has 160 vertices so it would be hard to draw it here. Note that $\diam G_1=1$ and $\diam G_2=3$. It can be easily shown that $\diam G_k=2k-1$.


Note that for any $x_t\in G_{k+1}$ we have $N_{G_{k+1}}(x_t)\subseteq N_{G_k}(x)\cup\{x\}$. So using Lemma \ref{DecompositionLemma} and a simple inductive argument we obtain that each $G_k$ is strongly finitely retractable and there is an evolution from $G_k$ to $G_{k+1}$, say \begin{tikzcd}G_k\arrow[r, rightsquigarrow, "f_k"]& G_{k+1}\end{tikzcd}.  

\begin{lm}\label{1PointExtensionEmbeds}
	Let $G$ be a subgraph of $G_{k}$. 
	Let $H$ be a one-point extension of $G$, i.e. $H\setminus G=\{x\}$, and let $\map rHG$ be a retraction. Then there is a point $x'\in G_{k+1}\setminus G_k$ such that $p_{k}(x')=r(x)$. Moreover, $\map fH{G\cup\{x'\}}$ given by $f(x)=x'$, and $f(t)=t$ otherwise, is a graph isomorphism. 
\end{lm}

\begin{proof}
	By Lemma \ref{1PointExtension}, $N_H(x)\subseteq N_{G}(r(x))\cup\{r(x)\}\subseteq N_{G_k}(r(x))\cup\{r(x)\}$. By the construction of $G_{k+1}$ there is $x'\in G_{k+1}\setminus G_k$ such that $x'$ is adjacent to all vertices of $N_H(x)$, and only to them, and $p_k(x')=r(x)$. Clearly $f$ is an isomorphism. 
\end{proof}

\begin{tw}
	Given $f_0, f_1, \ldots$ as above,
	\begin{equation}\label{EvolutionSequenceConstruction}
		\begin{tikzcd}
			G_0\arrow[r, rightsquigarrow, "f_{0}"]& G_1\arrow[r,rightsquigarrow, "f_{1}"]& G_2\arrow[r,rightsquigarrow,"f_{2}"]& \dots
		\end{tikzcd} 
	\end{equation}
	is a \fra\ sequence in the class of all connected and point-by-point retractable finite graphs.  
\end{tw}

\begin{proof}
	Assume that $H$ is connected and strongly finitely retractable finite graph and there is a chain
	\[
	\begin{tikzcd}
		G_k=H_0&\arrow[l, "r_{0}"'] H_1&\arrow[l, "r_{1}"']H_2&\arrow[l,]\cdots&\arrow[l,"r_{n-1}"'] H_n=H
	\end{tikzcd}
	\]
	such that $H_{i+1}\setminus H_i=\{x_i\}$ is a singleton and $p_i$ is a retraction. 
	By Lemma \ref{1PointExtensionEmbeds} and simple induction we may assume that $H$ is an induced subgraph of $G_{k+n}$ and $r_i(x_{i})=p_{k+i}(x_{i})$, that is $r_i$ is a restriction of $p_{k+i}$ to $H_{i+1}$. Moreover $x_i\in G_{k+i+1}\setminus G_{k+i}$ and $N_{G_{k+i+1}}(x_i)=N_{H_{i+1}}(x_i)$.
	
	
	Our aim is to define a chain of one-vertex extensions
	\[
	\begin{tikzcd}
		H=F_m&\arrow[l, "g_{m}"'] F_{m-1}&\arrow[l, "g_{m-1}"']F_{m-2}&\arrow[l,]\cdots&\arrow[l,"g_{1}"'] F_0=G_{k+n}
	\end{tikzcd}
	\]
	where $g_i$ are retractions such that 
	\[
	r_0\circ r_1\circ\cdots\circ r_{n-1}\circ g_1\circ g_2\circ\cdots\circ g_{m}=p^{k+n}_{k}.
	\]
	Firstly we remove, one by one, from $G_{k+n}$ all vertices from $G_{k+n}\setminus (G_{k+n-1}\cup\{x_{n-1}\})$. Let $\{z_1,\dots,z_{\ell_1}\}$ be an enumeration of $G_{k+n}\setminus (G_{k+n-1}\cup\{x_{n-1}\})$. Let $i\in\{1,2,\dots,\ell_1\}$. In the $i$-th step we put $F_i:=F_{i-1}\setminus\{z_i\}$ and define $\map {g_i}{F_{i-1}}{F_i}$ by $g_i(z_i)=p_{k+n-1}(z_i)$ and $g_i(z)=z$ for $z\in F_i$. By the construction $\{z_i,x_{n-1}\}$ is not a vertex in $G_{k+n}$, and therefore $g_i$ preserves the edges.
	
	Secondly we remove from $G_{k+n-1}$ all vertices from $G_{k+n}\setminus (G_{k+n-1}\cup\{x_{n-1}\})$. Let $\{z_{\ell_1+1},\dots,z_{\ell_2}\}$ be an enumeration of $G_{k+n-1}\setminus (G_{k+n-2}\cup\{x_{n-2}\})$. Let $i\in\{\ell_1+1,\dots,\ell_2\}$. In the $i$-th step we put $F_i=F_{i-1}\setminus\{z_i\}$ and define $\map {g_i}{F_{i-1}}{F_i}$ by $g_i(z_i)=p_{k+n-2}(z_i)$ and $g_i(z)=z$ for $z\in F_i$. By the construction $\{z_i,x_{n-2}\}$ is not a vertex in $G_{k+n-1}$. Suppose that $\{z_i,x_{n-1}\}$ is a vertex in $G_{k+n}$. Since $N_{G_{k+n}}(x_{n-1})=N_{H_{n}}(x_{n-1})$, then $\{z_i,x_{n-1}\}$ is a vertex in $H$, which contradicts the fact that $z_i\notin H$. Therefore $g_i$ preserves the edges.  
	
	Proceeding inductively we define all $g_i$'s. This completes the proof that \eqref{EvolutionSequenceConstruction} is a \fra\; sequence in the class of all connected PPRg finite graphs.
\end{proof}

The construction of a \fra\; sequence in the class of all finite (not necessary connected) PPR graphs is very similar. We start from the graph $\hat{G}_1:$\vspace{0.5cm}

\begin{tikzpicture}[every node/.style={circle,inner sep=2pt,fill=black}]
	\node[color=white] (X) at (0,1) {};
	\node[color=blue] (A) at (6,1) {};
	\node[color=red] (B) at (1,1) {};
	\node (C) at (4,1) {};
	
	\draw (B) -- (C);
\end{tikzpicture}
\vspace{0.5cm}\\
and in the inductive construction for every vertex $x$ in $\hat{G}_i$ and its neighborhood $N(x)=\{y_1,\dots,y_n\}$ we add, as before, new vertices $x_t$ for $t\in\{0,1\}^{n+1}$ where the vertex $x_{(0,\dots,0)}$ is isolated. Thus $\hat{G}_2$ is of the form:
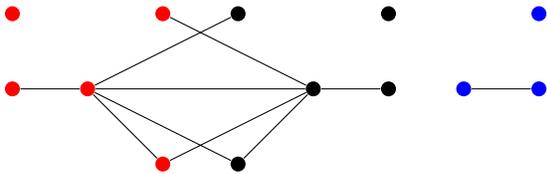
\vspace{0.5cm}

\begin{tikzpicture}[every node/.style={circle,inner sep=2pt,fill=black}]
	\node[color=red] (A) at (0,1) {};
	\node[color=red] (A1) at (0,2) {};  
	\node[color=red] (B) at (1,1) {};
	\node (C) at (4,1) {};
	\node[color=red] (D) at (2,0) {};
	\node (E) at (3,0) {};
	\node[color=red] (F) at (2,2) {};
	\node (G) at (3,2) {};
	\node (H) at (5,1) {};
	\node (H1) at (5,2) {};
	\node[color=blue] (I) at (6,1) {};
	\node[color=blue] (J) at (7,1) {};
	\node[color=blue] (J2) at (7,2) {};
	
	\draw (A) -- (B)
	(B) -- (C)
	(B) -- (D)
	(B) -- (E)
	(B) -- (G)
	(C) -- (D)
	(C) -- (E)
	(C) -- (F)
	(C) -- (H)
	(I) -- (J);
\end{tikzpicture}\\
As before, $\hat{p}_{2}$ maps vertices in $\hat{G}_2$ to vertices in $\hat{G}_1$ of the same color.

\section{Final remarks}

The goal of our study in this note was two-fold. First, developing graphs obtained by iterating special transitions, namely, those that add a vertex while keeping the information where it came from (encoded by a retraction). Second, studying profinite graphs that arise naturally in the process where both embeddings and retractions take an active part.

On the ``discrete" side, we have found several universal graphs (\fra\ limits of the corresponding categories), namely,
$\ufr$, $\ufrctd$, $\uppr$, $\upprctd$, $\soc$. These graphs are pairwise non-isomorphic (evidence is provided by their finite retracts), and all of them are universal, as the random graph is sociable, therefore PPR, connected, and finitely retractable.
It might be interesting to study properties of the automorphism groups of the graphs above.

On the profinite side, we have shown that neither of the natural envelopes is homogeneous, there is always a dense $G_\delta$ set of isolated vertices. Nevertheless, still the automorphism groups may be of interest, as possibly new Polish groups, not arising from classical \fra\ theory.

There is also the two-sided approach, where both embeddings and retractions are treated equally, and the \fra\ limit is a pair consisting of a profinite graph and its countable dense induced subgraph. This makes restrictions to the automorphism group, as now it consists of those topological isomorphisms that preserve the fixed countable dense subgraph.

We finish with the following open problem.

\begin{question}
	Is an infinite retract of a PPR structure again PPR? If not, then how about graphs? How about infinite retracts of sociable graphs?
\end{question}

\paragraph{Acknowledgments.} The authors would like to thank Adam Barto\v{s}, Tristan Bice, Paulina Radecka, and Paul Szeptycki for useful discussions on the topic of this note.


\end{document}